\documentclass[11pt, reqno]{amsart}

\usepackage[letterpaper, hmargin=1.0in, vmargin=1.0in]{geometry}
\usepackage{caption}

\usepackage[scale=1]{inconsolata}
\usepackage{bm}

\renewcommand{\leq}{\leqslant}
\renewcommand{\geq}{\geqslant}

\usepackage[letterpaper, hmargin=1.0in, vmargin=1.0in]{geometry}
\usepackage{import}
\usepackage{xifthen}
\usepackage{pdfpages}
\usepackage{transparent}
\usepackage{cancel}
\usepackage[normalem]{ulem}

\usepackage{tikz}
\usepackage{float}
\usepackage{mathtools, amsmath, amsfonts, amssymb, amsthm, mathrsfs}
\usepackage{pgffor}
\usepackage{etoolbox}
\usepackage[shortlabels]{enumitem}
\usepackage{xcolor}
\definecolor{CustomBlue}{RGB}{23, 86, 118}
\usepackage[unicode=true, colorlinks=true, linkcolor=CustomBlue, citecolor=CustomBlue]{hyperref}
\hypersetup{pdfauthor={Name}}
\usepackage{setspace}

\usepackage[T1]{fontenc}
\newcommand{\sse}{\subseteq}

\newcommand{\R}{\mathbb{R}}

\newcommand{\N}{\mathbb{N}}

\newcommand{\Z}{\mathbb{Z}}
\newcommand{\eps}{\varepsilon}

\newcommand{\lam}{\lambda}
\newcommand{\Om}{\Omega}

\newcommand{\sig}{\sigma}
\newcommand{\g}{\gamma}
\newcommand{\del}{\delta}

\newcommand{\al}{\alpha}
\newcommand{\be}{\beta}

\newcommand{\tr}{\mathrm{tr}}

\newcommand{\Nor}[1]{\mathcal{N}\left(#1\right)}
\newcommand{\Exp}[1]{\exp\left(#1\right)}
\newcommand{\Abs}[1]{\left|#1\right|}

\renewcommand{\Box}[1]{\left[#1\right]}
\newcommand{\Rnd}[1]{\left(#1\right)}
\newcommand{\Bra}[1]{\left\{#1\right\}}
\renewcommand{\bar}[1]{\overline{#1}}

\makeatletter
\newcommand{\E}[1][\@nil]{%
	\def\tmp{#1}%
	\ifx\tmp\@nnil
		\mathbb{E}
	\else
		\mathbb{E}\left[#1\right]
\fi}

\renewcommand{\P}[1][\@nil]{%
	\def\tmp{#1}%
	\ifx\tmp\@nnil
		\mathbb{P}
	\else
		\mathbb{P}\Rnd{#1}
\fi}
\makeatother

\newcommand{\Cov}{\mathrm{Cov}}
\newcommand{\Var}{\mathrm{Var}}

\newcommand{\norm}[1]{\left\lVert#1\right\rVert}
\newcommand{\mnorm}[1]{{\left\vert\kern-0.25ex\left\vert\kern-0.25ex\left\vert #1 
\right\vert\kern-0.25ex\right\vert\kern-0.25ex\right\vert}}

\newcommand{\1}{\mathbf{1}}

\renewcommand{\S}{\mathcal{S}}

\newcommand{\cP}{\mathcal{P}}

\newcommand{\f}[2]{\frac{#1}{#2}}
\newcommand{\F}[2]{{\left(\frac{#1}{#2}\right)}}

\newcommand{\df}{\coloneq}

\newcommand{\kl}[2]{d_{\mathrm{KL}}\Rnd{#1 || #2}}

\newcommand{\poly}{\mathrm{poly}}

\newcommand{\dt}{\mathrm{d}t}

\newcommand{\les}{\prec}

\usepackage{mathtools}
\mathtoolsset{showonlyrefs}

\newtheorem{theorem}{Theorem}[section]
\newtheorem{lemma}[theorem]{Lemma}

\theoremstyle{definition}
\newtheorem{remark}[theorem]{Remark}

\numberwithin{theorem}{section}

\numberwithin{equation}{section}
\numberwithin{figure}{section}

\newcommand{\cK}{\ensuremath{\mathcal K}}

\newcommand{\cN}{\ensuremath{\mathcal N}}

\newcommand{\cS}{\ensuremath{\mathcal S}} 
\newcommand{\cT}{\ensuremath{\mathcal T}}

\newcommand{\bC}{\mathbb{C}}

\def\({\left(}
\def\){\right)}

\def\d{\mathrm{d}}
\def\bul{$\bullet$\hskip2mm }
\def\i{\iota}

\newcommand{\TV}[1]{\left\| #1 \right\|_{\mathrm{TV}}}

\def\ss{{s_\star}}

\def\G{\Gamma}

\def\lm{\lam_{\mathrm{min}}}
\def\les{\lesssim}
\newcommand{\Inn}[1]{\left\langle #1 \right\rangle}
\renewcommand{\Re}[1]{\mathrm{Re}\Rnd{#1}}

\let\temp\phi
\let\phi\varphi
\let\varphi\temp

\def\go{g^{(1)}}
\def\gt{g^{(2)}}
\def\ao{a^{(1)}}
\def\at{a^{(2)}}
\def\muan{\mu_{n, \g, a}}
\def\mua{\mu_{\g, a}}
\def\mugn{\mu_{n, \g, g}}
\def\mug{\mu_{\g, g}}
\def\mutan{\widetilde{\mu}_{n, \g, a}}
\def\muta{\widetilde{\mu}_{\g, a}}

\def\mutg{\widetilde{\mu}_{\g, g}}
\def\tS{\widetilde{S}}
\def\tZ{\widetilde{Z}}
\def\cKn{\cK_{n, \g}}
\def\cKa{\cK_{n, \g, a}}
\def\cKg{\cK_{n, \g, g}}

\begin{document}
\title[Invariance principle for GMC]{Invariance principle for the Gaussian multiplicative chaos via a high dimensional CLT with low rank increments}

 \author{Mriganka Basu Roy Chowdhury, Shirshendu Ganguly }
 \address{Mriganka Basu Roy Chowdhury\\ University of California, Berkeley}
 \email{mriganka\_brc@berkeley.edu}
 \address{Shirshendu Ganguly\\ University of California, Berkeley}
 \email{sganguly@berkeley.edu}

\begin{abstract} Gaussian multiplicative chaos (GMC) is a canonical random
	fractal measure obtained by exponentiating log-correlated Gaussian
	processes, first constructed in the seminal work of Kahane \cite{kahane1985chaos}. Since
	then it has served as an important building block in constructions of
	quantum field theories and Liouville quantum gravity. However, in many
	natural settings, non-Gaussian log-correlated processes arise. In this paper,
	we investigate the universality of GMC through an invariance principle. We consider the
	model of a random Fourier series, a process known to be log-correlated.
	While the Gaussian Fourier series has been a classical object of study,
	recently, the non-Gaussian counterpart was investigated and the associated
	multiplicative chaos constructed  in \cite{junnila2020multiplicative}. We
	show that the Gaussian and non-Gaussian variables can be coupled so that
	the associated chaos measures are almost surely mutually absolutely
	continuous \textit{throughout the entire
	sub-critical regime}. This solves the main open problem from \cite{kk} who had earlier established such a result for a part of the regime. The main ingredient is a new high dimensional CLT for
	a sum of independent (but not i.i.d.) random vectors belonging to rank one subspaces with error bounds involving the isotropic properties of the
	covariance matrix of the sum, which we expect will find other applications. The proof
	relies on a path-wise analysis of Skorokhod embeddings as well as a
	perturbative result about square roots of positive semi-definite matrices  which, surprisingly, appears to be new. 
\end{abstract}
\maketitle{}

\begin{figure}[h]
	\centering
	\includegraphics[width=0.87\textwidth]{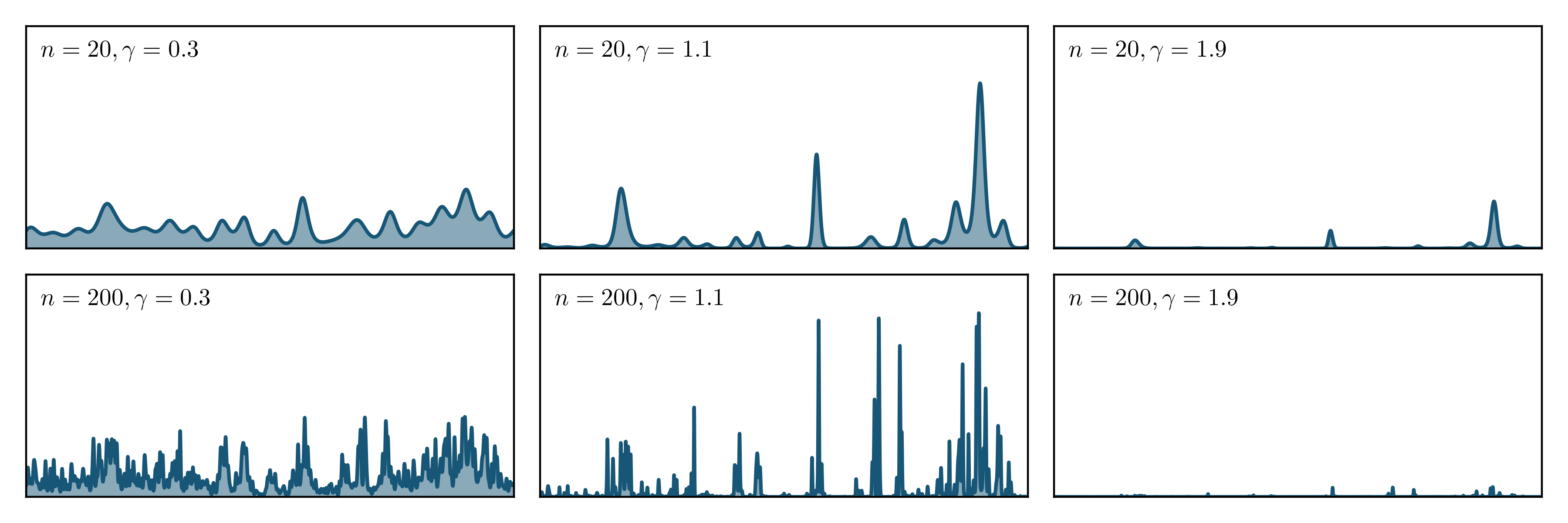}
	\captionsetup{singlelinecheck=off}
	\caption*{
		\small
		{\scshape Figure.}
		Plots of the (prelimiting) GMC density associated with the random Fourier series
		$$
			[0, 1] \ni t \mapsto \sum_{k = 1}^n k^{-1/2} \Rnd{g_k^{(1)} \cos(2\pi kt) + g_k^{(2)}\sin(2\pi kt)}, \quad g_k^{(1)}, g_k^{(2)} \sim \cN(0, 1).
		$$
		The six plots show the densities for varying values of the number of terms $n$ (varied along
		columns), and intermittency parameter $\g$ (varied along rows). The
		density is proportional to the \emph{exponential}
		of ($\g$ times) the Fourier series, properly normalized. The GMC for a fixed $\g$ is
		the large-$n$ limit of the measures with these densities. Note that the measure
		is much rougher and the support much smaller for larger values of $\g < \sqrt{2}$.
		For $\g \geq \sqrt{2}$, the support vanishes in the $n \to \infty$ limit, and thus the GMC
		is trivial, as illustrated in the plots corresponding to $\g = 1.9$ (last column).
	}
	\label{fig:gmc}
\end{figure}
\newpage

\setcounter{tocdepth}{2}
\tableofcontents

\parindent=0pt
\parskip=5pt

\section{Introduction and main results}

The paper deals with two seemingly different directions of modern probability and uses one to study the other.
The first is theory of Gaussian multiplicative chaos.

\subsection{Gaussian multiplicative chaos (GMC)}
This was introduced by Kahane in \cite{kahane1985chaos} as a mathematical model for energy dissipation in turbulence,
making rigorous a program initiated by Mandelbrot in \cite{mandelbrot2005possible}. 
More precisely, Kahane gave a
rigorous interpretation of the measure
$$
M_\gamma(dx) := e^{\gamma X(x)-
\frac{\gamma^2}{2} \Var[X(x)]}\sigma(\d x)
$$ 
where $(X(x))_{x\in T}$ denotes a log-correlated Gaussian field indexed by an arbitrary metric space $T$ equipped with a
measure $\sigma(\d x)$.
Here log-correlated indicates that the covariance $R(x,y)$ decays logarithmically in the distance $d(x,y)$.

Log-correlated Gaussian fields (LCGF) are expected to occur as the fluctuation theory in many natural critical
statistical physics models with the two-dimensional Gaussian free field being perhaps the most canonical such example
\cite{bp}.  The analysis of extreme values of log-correlated fields has been an important topic of study (see e.g., \cite{bdzgff,z, biskup}).
GMCs, the fractal random measures obtained by exponentiating such LCGFs, have turned out to be central building blocks of
quantum field theories and Liouville quantum gravity, beyond being of independent interest (see
\cite{rhodes2014gaussian} for a beautiful survey on this topic).  A particular feature of such measures, which will also be
crucially used in our work as well, is that they are supported on the set of \emph{thick points} of the underlying field, i.e., points whose
values are close to that of the maximum.

On the other hand, there are various examples of non-Gaussian log-correlated
fields producing a limiting \emph{Gaussian} multiplicative chaos such as those arising in random matrix theory
as log-characteristic polynomials of unitary matrix ensembles
\cite{webb2015characteristic}. Models with a number-theoretic flavor, expected
to behave similarly, consider the restriction of the Riemann-zeta function on a
random interval on the critical line \cite{saksman2020riemann}.

Only very recently, a general approach to studying non-Gaussian log-correlated fields and associated multiplicative
chaos measures have been initiated in \cite{junnila2020multiplicative} where the author considered a class of models with a
martingale structure and developed a general result under certain moment assumptions to establish the existence and
convergence to the corresponding multiplicative chaos measures. A particular example being the random Fourier series
model whose definition appears shortly. While this is the model of consideration in this article, other interesting
examples include the Brownian multiplicative chaos obtained by formally exponentiating the square-root of the
occupation field of 2D Brownian motion (see e.g. \cite{jegogmc,jegobmc, jegopbmc,abjl}). We now arrive at the central object of interest in this article.  For a
sequence of i.i.d. random variables $a=\big\{\ao_k, \at_k \big\}_{k\ge 1}$ let 
	\[
		S_{n, a}(t) \df \sum_{k = 1}^n \f1{\sqrt{k}} \Rnd{\ao_k \cos(2\pi kt) + \at_k \sin(2\pi k t)}, \quad t \in [0, 1].
	\]
	When the disorder variables are i.i.d. standard Gaussians, we denote the
	analogous object $S_{n, g}(t).$ Note that the variance of $S_{n, a}(t)$
	grows like $\log n$ and hence this sequence of functions do not converge in
	the sense of functions. Nonetheless, there is a limit $S_{\infty, a}(t)$ in
	the space of tempered distributions, which, while non-Gaussian, is also
	log-correlated. In \cite{junnila2020multiplicative}, using martingale techniques, Junnila also
	constructed the corresponding multiplicative chaos and exhibited its
	non-triviality for all $\gamma \in (0,\sqrt 2)$ analogous to the Gaussian
	case (we will indicate a heuristic reason explaining why the measure
	degenerates for larger values of $\gamma$ in Section \ref{iop}). Subsequently, in \cite{kk}, the
	question of an invariance principle in this context was investigated.
	Namely, whether the multiplicative chaos corresponding to $S_{\infty,
	a}(t)$ is related to that corresponding to $S_{\infty, g}(t).$  Note that
	they cannot be exactly the same since the $k=1$ term in the sum has a
	global multiplicative non-Gaussian effect. However, in \cite{kk}, the authors
	showed that the disorder variables $a$ and $g$ can be coupled so that the
	resulting chaos measures are absolutely continuous with respect to each
	other as long as $\gamma \in (1,\sqrt 2).$ It turns out that this is a
	natural threshold for this problem since the associated multiplicative
	martingale given by the measures  $e^{\gamma S_{n, a}(t)-
	\frac{\gamma^2}{2} \Var[S_{n, a}(t)]}\d t$ converges in $L^2$ for $\gamma
	<1$ and only in $L^1$ for $\gamma \in [1,\sqrt 2).$ While traditionally,
	the $L^2$ regime has been easier to analyze considering one can make
	arguments based on the second moment, in this case, surprisingly, the
	results of \cite{kk} only hold in the interior of the $L^1$ regime, essentially on account of
	the fact that the support of the GMC measure is ``too big'' when $\gamma \le 1.$
	This left open the question of treating the regime $\gamma \in
	(0,1]$, with no clear conjecture on the actual truth in this regime. 

In this article we prove the universality of GMC for all $\gamma \in (0,\sqrt
2)$ in the above sense, via proving a result from high dimensional probability
that we expect will have much wider applications. This leads us to the world of
high dimensional central limit theorems. 
\subsection{High dimensional CLT}
The usual classical central limit theorem says that for any fixed $d,$ if $X_1,
X_2, \ldots, X_n$	are i.i.d. $d-$dimensional vectors with mean zero and
covariance matrix $\Sigma$, then 
\begin{align}\label{keyobject}
\frac{1}{\sqrt n}\sum_{i=1}^{n} X_i\overset{d}{\to} Z
\end{align}
where $Z\sim N(0,\Sigma).$
While the one (or fixed) dimensional  case is very well understood at this
point (see e.g. \cite{berry1941accuracy, rio2009upper,
rio2011asymptotic, bobkov2013entropic, bobkov2018berry}), in many applications
one has $d$ growing with $n$
leading to interesting questions about convergence rates as a function of $d$
and $n$. In several applications, one also witnesses phase transitions where  a
CLT holds only when $d$ is less than a certain threshold function of $n.$ 

Multidimensional central limit theorems have been studied extensively since the mid-twentieth century  
\cite{bergstrom1945central} (see also \cite{bhattacharya1977refinements} and
references therein). Considering probabilities of convex
sets, the dependence of the convergence rate on the dimension was 
studied in a long list of works (see for instance, \cite{nagaev2006estimate, senatov1981uniform, gotze1991rate, bentkus2005lyapunov, chen2011multivariate}).

This broad topic has been the test-bed for various ideas and approaches.
Entropy based methods provide a robust framework in this regard, going back to \cite{barron1986entropy} where it was shown, in the one dimensional case, that convergence in relative entropy
to a Gaussian occurs as long as the distribution of the individual variables has finite
relative entropy with respect to the Gaussian.
An optimal convergence rate in relative entropy for distributions having a spectral gap
in one-dimensions was established in\cite{artstein2004rate}
and \cite{johnson2004fisher}. A later result in \cite{bobkov1, bobkov2} using Edgeworth-type expansions also applied to higher dimensions albeit with exponential dependence on dimension. 
Over the past few years there have been several advances in the study of high
dimensional central limit theorems with interesting geometric applications. One
particular instance involves testing whether a sample covariance matrix,
properly centered and scaled, is close to a symmetric Gaussian matrix, which is
related to testing the presence of latent geometry in random graph models
\cite{bder}. In \cite{bubeck2018entropic}, Bubeck and the second named author developed an entropy
based argument to establish a sharp transition for this problem (see also the
work of Eldan and Mikulincer \cite{eldan2020information}).  There have also been developments in the study
of maxima of sums of independent random vectors (see \cite{chernozhukov2013gaussian}).

More recently, two papers deserve particular mention, with the second one
related to the ideas developed in this paper. \cite{zhai2018high} considered the
$L^2$ transportation distance (denoted by $W_2$) between the LHS and RHS in
\eqref{keyobject} and proved an almost optimal dimension dependent bound (up to
logarithmic factors). The proof follows a Lindeberg-type strategy of gradually
replacing $X_i$s with Gaussians and using as key input 
Talagrand’s transportation inequality.  The argument relies on bounding
$W_2(\sqrt{n} Z, \sqrt{n-1}Z+X_n)$ where $Z\sim N(0,\Sigma)$ for some
covariance matrix $\Sigma$, whereas $X_n$ is a general distribution with mean
zero and the same covariance matrix.

However, in our setting the increments $X_i$s will be far from i.i.d. and in fact they will have rank one covariances. 
Thus, the
plausibility of a CLT relies on a global homogenization taking place ruling out a Lindeberg-type approach. 

In a subsequent work, Eldan, Mikulincer, and Zhai \cite{eldan2020clt} used
martingale embeddings constructed by Eldan \cite{eldan} to obtain several new bounds
for convergence in transportation distance and entropy, improving, among other
things, the logarithmic factors in Zhai's result. They also obtained strong
results under the additional assumption of log-concavity.

The use of Skorokhod embeddings has served as a powerful tool to prove CLTs going back to Strassen's original work on
the one dimensional CLT \cite{st}.
At a high level, we will use Skorokhod embeddings as well to prove a new high
dimensional CLT with rank one increments which will then be applied to prove
the previously claimed universality of the GMC. Several further ingredients
will feature, including a seemingly new perturbation result on matrix square roots which
will be reviewed in Section \ref{psdsqrt}. The CLT result is expected to be of wider
interest. Leaving further review and comparisons to earlier results for later,
we now move on to the statement of our main results. 

Our first main result states the universality of the GMC in the sense of \cite{kk} for the random Fourier series model. 
Suppose $\Bra{\go_i, \gt_i}_{i \geq 1}$ are i.i.d. standard Gaussian random variables,
	and $\Bra{\ao_i, \at_i}_{i \geq 1}$ are i.i.d. random variables sampled from a common law $\mu$, satisfying 
	\begin{align}
		\label{eq:acond}
		\E a = 0, \quad \Var(a) = 1, \quad \E[e^{\lam a}] < \infty,
	\end{align}
	for $a \sim \mu$ and $\lam \in \R$.
Recall that
\begin{align}\label{fourier1}
		S_{n, a}(t) \df \sum_{k = 1}^n \f1{\sqrt{k}} \Rnd{\ao_k \cos(2\pi kt) + \at_k \sin(2\pi k t)}, \quad t \in [0, 1],
\end{align}
	and the analogously defined $S_{n, g}(t)$. Further, define the (prelimiting) random measures on $[0,1]$
	\begin{align}\label{gmc1}
		\muan(\d t) \df \f{e^{\g S_{n, a}(t)}}{Z_{n, \g, a}} \d t, \quad Z_{n, \g, a} \df \E e^{\g S_{n, a}(t)},
	\end{align}
	While it is known since the work of Kahane \cite{kahane1985chaos}, that
	$\mugn \to \mug$, the latter referred to as the \emph{Gaussian
	multiplicative chaos}, recently, in \cite{junnila2020multiplicative}, it
	was shown that $\muan \to \mua$ weakly almost surely to some multiplicative
	chaos measure $\mua$ as well. 
	
	\begin{theorem}[Universality of GMC]
	\label{thm:gmc}
	Let $\g \in (0, \sqrt 2)$. 
	Then, there is a coupling of $\Bra{\ao_i, \at_i}_{i \geq 1}$ and $\Bra{\go_i, \gt_i}_{i \geq 1}$ such that
	\[
		\mug \ll \mua \ll \mug,
	\]
	that is, the multiplicative chaos measures are mutually absolutely continuous.
\end{theorem}

Natural higher dimensional analogues of the above theorem  hold as well. We elaborate more on this in Section \ref{gdom} later. 

As already mentioned, this solves the main open question in \cite{kk} where the authors had proved the above result for all $\gamma \in (1,\sqrt 2).$ We will
review their proof strategy in more detail but for the moment we move on to the next result, which will serve as a key
input for us. We expect this new high dimensional central limit theorem to find other applications as well.

	Suppose $v_1, v_2, \ldots, v_n \in \R^d$ are $n$ vectors (where $d \leq n$) with $\norm{v_i}_2  \leq \sqrt{d}$ for
	all $i$, (in fact in our application we will have the stronger condition $\norm{v_i}_\infty  \leq 1$),  and
	let $\mu$ be a distribution on $\R$ with mean 0, variance 1, and with stretched-exponential tails, i.e.,
	there are constants $C, c, \alpha > 0$ such that if $a \sim \mu$, then
	\begin{align}\label{tailass}
		\P(|a| \geq t) \leq C \exp(-c t^\alpha), \quad \text{for all $t \geq 0$}.
	\end{align}
	Define the normalized sums
	\begin{align*}
		S^a_n &= \f{1}{\sqrt n} \sum_{i=1}^n a_i v_i, \quad a_1, a_2, \ldots, a_n \sim \mu, \text{i.i.d.}, \\
		S^g_n &= \f{1}{\sqrt n} \sum_{i=1}^n g_i v_i, \quad g_1, g_2, \ldots, g_n \sim \Nor{0, 1}, \text{i.i.d.}.
	\end{align*}

\begin{theorem}
	\label{thm:main}
There is a coupling of $S^a_n$ and $S^g_n$ (and thus, of the collections $\{a_i\}_i$ and $\{g_i\}_i$)
	such that for any $r \geq 1$, we have
	\[
		\norm{S^a_n - S^g_n}_\infty \les r \log^{1 + \f1\be} n \cdot \max\Rnd{\F{d\norm{U}}{n}^{1/4}, \F{d}{n}^{1/2}},
	\]
	with probability at least $1 - O(n^{-r})$. Here, $\be = \f{\al}{2 + \al}$ (see Remark \ref{rem:setails} below) and
	$\norm{U}$ refers to the spectral/operator norm of the matrix
	$U$ defined as
	\[
		U \df \f1n \sum_{i = 1}^n v_i v_i^T.
	\]
	Moreover, $\lesssim$ denotes, as usual, that there is a universal constant which when multiplied to the RHS makes the LHS smaller.
\end{theorem}
A few remarks are in order. First, it is natural to wonder about the necessity of the various hypotheses in the above
statement. The $d\le n$ hypothesis is
natural since one does need more randomness than dimensions to have a chance at homogenization per coordinate (a short
counterexample in the case $d=n$ is provided later in Remark \ref{rmk:dncounter}).  The stretched exponential tail assumption could
potentially be relaxed but, as will be explained shortly (see Remark \ref{rem:setails}), will indeed feature somewhat crucially in our arguments.  

Finally, while our setting is particularly singular compared to past results on account of the low rank of the
covariance matrices of the increments, it might be instructive to compare Theorem \ref{thm:main} to  the $W_2$ bounds
obtained in the already alluded to results in \cite{zhai2018high, eldan2020clt}. The latter, up to 
log-corrections is $\kappa\sqrt{\frac{d}{n}}$, where $\kappa$ is an assumed $L^2$ bound on the increment random vectors.
Translated to our setting yields a $W_2$ bound of $\frac{d}{\sqrt n}$ since  $\kappa \le \sqrt{d}$ by hypothesis in our case. Thus, note that this $W_2$  bound  is consistent with an
$L^{\infty}$ bound of $\sqrt{\frac{d}{n}}$ while we get a bound of order $(\frac{d}{n})^{1/4}$ ignoring smaller
correction factors, assuming $\|U\|$ is sub-polynomial in $n$. We are inclined to believe that indeed $\sqrt{\frac{d}{n}}$ is the optimal
$L^\infty$ bound and the suboptimal $1/4$ exponent stems from a perturbative bound on matrix square roots which we will elaborate more on shortly.

We provide a short outline of the key ideas next.

\subsection{Idea of the proof}\label{iop} We first discuss the proof of the central limit theorem since that involves
all the new ideas in the paper.  Given this, the proof of Theorem \ref{thm:gmc} essentially follows from the arguments
in \cite{kk}.

The main difficulty in proving Theorem \ref{thm:main} is that the increments
live in one dimensional subspaces and hence are far from isotropic. This rules
out an approach reliant on a
Lindeberg-type argument similar to the one implemented in \cite{zhai2018high}. Thus, to
prove a central limit theorem of the form of Theorem \ref{thm:main}, one needs to
capture the cancellations that occur globally induced by increments occurring
along a given direction multiple times across the sum. As already alluded to,
we employ a path-wise analysis to accomplish this inspired by the use of
martingale embeddings to derive CLTs going back to \cite{st} and more recently in the high dimensional case in \cite{eldan2020clt}. However, while most prior
works on CLT consider smoother metrics such as the Wasserstein $L^2$ metrics,
here, on account of the application required for Theorem \ref{thm:gmc}, we seek to work with the $L^{ \infty}$ norm. Further, the rank one nature
of the increments is in stark contrast to more
commonly analyzed settings where the increments are more isotropic.

We will use Skorokhod embeddings to use a common set of Brownian motions to embed both the variables $a_i$ and  $g_i$.
Let $B^1,B^2,\ldots ,B^n$ be $n$ independent Brownian motions with $\tau_i$
be i.i.d. stopping times such that $B^i_{\tau_i}=a_i.$ Then considering
$X^i_t$, the stopped process $v_i B^i_{t \wedge \tau_i},$  observe that $S_n^a=\frac{1}{\sqrt n}\sum_{i=1}^n X^i_\infty.$
Now,
\begin{align}
\d X^i =v_i \1_{t < \tau_i} \d B^i ,
\end{align}
and the quadratic variation process is $\d {[X^i]}_t=v_iv_i^T \1_{t < \tau_i}.$
By a change of variable, the process $\frac{1}{\sqrt n}\sum_{i=1}^n X^i_t$ has
the {differential} representation $\sqrt{\frac{1}{n}\sum_{i=1}^n v_iv_i^T \1_{t < \tau_i}} \d W_t$  
where $W_t$ is a standard $d-$dimensional Brownian motion. We now compare
it with the ``averaged'' process $$\sqrt{\frac{1}{n}\sum_{i=1}^n v_iv_i^T \P(t <
\tau_i)} \d W_t$$ whose infinite-time limit is indeed an embedding of the Gaussian
vector $S_n^g$ which can be checked by computing the limiting quadratic
variation.  Thus, this reduces the problem to the analysis of the matrix process
$$\sqrt{\frac{1}{n}\sum_{i=1}^n v_iv_i^T \1_{t < \tau_i}}
-\sqrt{\frac{1}{n}\sum_{i=1}^n v_iv_i^T \P(t < \tau_i)}.$$ The matrix Bernstein
inequality implies that the matrices without the square root are close in
operator norm which then necessitates establishing a bound on the closeness of
the square roots of two positive semi-definite matrices which are themselves
close. While perturbative results of such kind have been the object of
extensive research, see for instance \cite{schmitt1992perturbation, perturb}, often such bounds rely on
estimates of the lowest singular value of such matrices which are unavailable
for the kind of matrices we will encounter. We prove a perturbative bound
for square roots of matrices analogous to what one has for real numbers, which is of independent interest. It seems, somewhat surprisingly, that such a bound hasn't appeared in the literature before.

Given the CLT result, the proof of the GMC universality result Theorem \ref{thm:gmc} follows the strategy in \cite{kk}. While a
more detailed outline will be presented in Section \ref{gmcpf}, for the moment let us simply  mention that the proof proceeds
by  discretizing the model into a hierarchical model, i.e., comparing the log-correlated process $S_{n, a}(t)$ to a
model on a tree (akin to how branching random walk is a, simpler to analyze, hierarchical proxy for the Gaussian free field), which reduces the problem to a finite dimensional one. We will then apply the CLT result Theorem
\ref{thm:main}
to compare the models with noise variables $a$ and the Gaussians $g$ on the tree. In the application we will take
$n=2^m$, and $d$ will be essentially the support size of the GMC at that scale which is roughly the set of $\gamma$
thick points and consequently has size approximately $n^{1-\frac{\gamma^{2}}{2}}$ which follows from a first moment
argument (note that the expression becomes negligible when $\gamma \ge \sqrt 2$ indicating the triviality of the GMC in this regime).  We will need the closeness bound delivered by the CLT result to be summable in $m.$ Note that it is
exponentially small in $m$ for all $\gamma\in (0,\sqrt 2)$ only when $\|U\|$ is sub-polynomial in $n$. However, the
worst case bound on $\|U\|$ given by $\|U\|\le \tr(U)\le d$, where the last inequality follows since $\|v_i\|^2_2=O(d)$ for all $i\le n$ by the hypothesis in Theorem
\ref{thm:main}), makes the error bound effective only when $\frac{d^2}{n}$
is small and consequently when $\gamma \in (1,\sqrt 2)$ recovering the result of \cite{kk}. Thus, to obtain the result for
all $\gamma$, we will prove that $\|U\|$ is poly-logarithmic in $n$, i.e., $U$ is approximately isotropic. At this point
we will crucially use that $U$ admits strong connections to a discrete
Fourier transform (DFT), matrix (see e.g., \url{https://en.wikipedia.org/wiki/DFT_matrix}).

\begin{remark}
	\label{rem:setails} An assumption that will feature crucially in our proof is a stretched exponential tail on the
	stopping time $\tau$	in our Skorokhod embedding of $\mu$. This piece is furnished by \cite[Theorem
	1]{sawyer1972remark}, which shows that a stretched exponential tail of parameter $\al$ for $\mu$ transfers to a
	stretched exponential tail of parameter $\be$ for (a particular) Skorokhod embedding stopping time $\tau$, where $\be = \f{\al}{2 +
	\al}$. See \cite[Section 2]{sawyer1972remark} for more details on the construction of such a $\tau$.
\end{remark}

\subsection{Organization}
The rest of the article is organized as follows.  The upcoming Section \ref{main1} is devoted to the proof of the CLT
result Theorem \ref{thm:main}. We want to draw particular attention to Section \ref{psdsqrt} where perturbation results
for square roots of positive semi-definite matrices are obtained. Given the CLT result, Section \ref{gmcpf} essentially
follows the strategy of \cite{kk} to prove Theorem \ref{thm:gmc}. Certain standard arguments are included in the
Appendix for completeness.

\subsection{Acknowledgements} SG  thanks Yujin Kim for mentioning the problem of absolute continuity of non-Gaussian
multiplicative chaos measures addressed in this article. The authors also thank
Nikhil Srivastava for helpful discussions on matrix perturbation results. SG
was partially supported by  NSF CAREER Grant-1945172. 

\section{High dimensional CLT: Proof of Theorem \ref{thm:main}}\label{main1}
For the ease of reading we will break the argument into several short subsections each serving a distinct purpose. 

\subsection{Construction of the coupling.}

Consider Brownian motions (in $\R$) $B^1, \ldots,
B^n$, and define the processes $X^i_t$ via $\d X^i_t = v_i \1_{t < \tau_i} \d B^i_t$ (with $X^i_0 = 0$), where $\tau_i$ is
the (a.s. finite) stopping time
associated with $B^i$ as given by Remark \ref{rem:setails} such that $B^i_{\tau_i} \sim \mu$. Then, clearly
$X^i_\infty$ has the same law as $v_i a_i$, and so
\[
	\Phi_\infty \sim \f{1}{\sqrt n} \sum_{i = 1}^n v_i a_i = S^a_n, \quad
			\text{where } \Phi_t \coloneq \f{1}{\sqrt n} \sum_{i = 1}^n X^i_t.
\]

Since $\d \Phi_t = \f{1}{\sqrt n} \sum_{i = 1}^n v_i \1_{t < \tau_i} \d B^i_t$,
its quadratic (co-)variation process is given by
\[
	\d [\Phi]_t = \f{1}{n} \sum_{i = 1}^n v_i v_i^T \1_{t < \tau_i} \d t = \G^2_t \dt, \quad \G_t \coloneq
	\sqrt{\f{1}{n} \sum_{i = 1}^n v_i v_i^T \1_{t < \tau_i}}.
\]
It will be convenient to introduce the compact notation $\d \Phi_t= V_t \d B_t$
where $B=(B^1,\cdots, B^n)$ and $V_t$ is the matrix with the $i^{th}$ column
$V^i_t$ equal to $\f{1}{\sqrt n}  v_i \1_{t < \tau_i}.$ At this point we apply
a change of variable justified by the following Lemma \ref{lem:diagonalization}
to alternatively express $\Phi_t$ as
\[
	\d \Phi_t = \G_t \d W_t
\]
for a Brownian motion $W$ on $\R^d$ (on a potentially extended probability space).

\begin{lemma}
	\label{lem:diagonalization}
	Consider a centered Ito process $X_t$ on $\R^d$ defined as $\d X_t = A_t \d B_t$, where $A_t \in \R^{d \times n}$ is a matrix-valued process
	adapted to the Brownian motion $B_t$ (in $\R^n$). Then, there exists a standard Brownian motion $W_t$ in $\R^d$ (on a potentially
	extended probability space) such that
	\[
		\d X_t = Q_t \d W_t, \quad Q_t \df \sqrt{A_t A_t^T},
	\]
\end{lemma}
If $Q_t$ is invertible, then $W_t$ is defined through $\d W_t=Q_t^{-1}A_t\d
B_t$ which is a martingale with quadratic variation process having differential
increment $I \d t.$ By Levy's characterization (see \cite[Theorem
3.16]{shreve}) this implies that this is a Brownian motion. In the case $Q_t$
is singular, the proof is only a bit more complicated and the details are
provided in the appendix.

Continuing with the proof of Theorem \ref{thm:main},
this representation suggests that a potential Gaussian approximator of this process
would try to match the infinitesimal covariance structure of $\Phi_t$ at each time $t$ via a deterministic proxy.
Consequently, it is natural to introduce a process $\Psi_t$ defined via
\[
	\d \Psi_t = G_t \d W_t, \quad \Psi_0 = 0, \quad \text{ where $G_t \coloneq \sqrt{\E\G_t^2}$}.
\]
It is clear that $\Psi_\infty$ is Gaussian, so it suffices to verify that it has the
same covariance structure as $S^g_n$, which is
\[
	\Cov\Rnd{\f{1}{\sqrt n} \sum_{i = 1}^n v_i g_i} = \f{1}{n} \sum_{i = 1}^n v_i v_i^T \eqcolon U.
\]
This follows from the SDE for $\Psi_t$ since
\begin{align*}
	\f{\d}{\d t}\Cov(\Psi_t) &= G_t G_t^T \\
						  &= \E\G_t^2 \\
						  &= \E[\f{1}{n}\sum_{i = 1}^n v_i v_i^T \1_{t < \tau_i}] \\
						  &= U p_t,
\end{align*}
where
\[
	p_t \coloneq \P(\tau > t).
\]
Then,
\[
	\Cov(\Psi_\infty) = U \int_0^\infty p_t \d t = U \E \tau = U.
\]
Thus, $\Psi_\infty$ has the same law as $S^g_n$, and therefore the pair $(\Phi_\infty, \Psi_\infty)$ is a valid coupling of $S^a_n$
and $S^g_n$. Note that the probability space supporting the coupling is the extended probability space on which $W_t$ is defined.

\subsection{Bounding the error per coordinate}
To establish the promised bound on the coupling error $\norm{S^a_n - S^g_n}_\infty$ in Theorem
\ref{thm:main}, we will investigate the gap per coordinate. This gap will turn out to have light tails, enabling a union
bound over all $d$ coordinates.

Without any loss of generality, let us consider the first coordinate of the gap process, i.e., $\Phi^1_t - \Psi^1_t$
(we will use superscripts to denote the coordinate).
By definition,
\begin{align*}
	\Phi^1_t - \Psi^1_t &= \int_0^t e_1^T (\G_s - G_s) \d W_s,
\end{align*}
where $e_i^T$ denotes the $i$-th coordinate vector in $\R^d$.
A useful tool to understand the tails of this process is to consider the quadratic variation, which is given by
\begin{align}
	\label{eq:quadvar}
	[\Phi^1 - \Psi^1]_t &= \int_0^t \norm{e_1^T(\G_s - G_s)}^2_2 \d s \leq \int_0^t \norm{\G_s - G_s}^2 \d s,
\end{align}
where, as before, we use the notation $\norm{\cdot}$ for the operator (spectral) norm of a matrix (note that in the
above inequality we use the straightforward bound $\norm{e_1^T A}_2^2 \le \norm{A}^2$ for any symmetric matrix $A$).
At this point, it is helpful to recall the form of $\G_s$ and $G_s$:
\begin{align}
	\label{eq:gammagdefs}
	\G_s &= \sqrt{\f{1}{n} \sum_{i = 1}^n v_i v_i^T \1_{s < \tau_i}}, \\
	G_s &= \sqrt{\E\G_s^2}.
\end{align}
Letting \begin{equation}\label{centeredindicator}
\chi^i_s= \1_{s < \tau_i} - p_s,
\end{equation} 
we may write
\begin{align*}
	\G_s &= \sqrt{\f{1}{n} \sum_{i = 1}^n v_i v_i^T \chi^i_s + p_s U},
\end{align*}
using the notation $U = \f{1}{n} \sum_{i = 1}^n v_i v_i^T$ introduced earlier. Since $G_s^2 = \E \G_s^2 = p_s U$, the
gap $\G_s - G_s$ is given by
\begin{align}
	\label{eq:gap}
	\G_s - G_s = \sqrt{A_s + p_s U} - \sqrt{p_s U}, \quad A_s \df \f{1}{n} \sum_{i = 1}^n v_i v_i^T \chi^i_s.
\end{align}

Now $A_s$ is an average of $n$ independent, mean zero matrices, which by matrix concentration estimates will have small
operator norm.   Now observe that if $A_s$ is small in operator norm, we expect $\sqrt{A_s + p_s U} \approx
\sqrt{p_s U}$ allowing us to bound $\norm{\G_s - G_s}.$
This leads us to consider perturbations of square roots of positive semi-definite matrices. The upcoming section is
devoted to this topic which, beyond the applications in this paper, is expected to be of broader interest.

\subsection{Perturbations of square roots of positive semi-definite matrices}\label{psdsqrt}
{General results probing the spectral properties of the  gap $\sqrt{P + Q} - \sqrt{P}$, where $P$ is a positive semi-definite matrix and $Q$ is a
	perturbation matrix have a significant history, but they all seem to rely on the largeness of the {\it lowest
	eigenvalue of $P$} (see
\cite{schmitt1992perturbation} for instance)
a quantity that is often hard to control and is expected (through numerical experiments) to be particularly small in our
setting.} Note on the other hand that for non-negative real numbers:
\begin{align}\label{realsqrt}
|\sqrt{p + q}-\sqrt p| \le \sqrt q
\end{align}
independent of $p$, while if $p\gg q,$ then the above bound can be improved to $O(\frac{q}{\sqrt p}).$ Since only a
bound of the latter kind is what seems to be available in the literature, one of our key inputs is the following result
analogous to \eqref{realsqrt}.

\begin{lemma}[Square-root gap]
	\label{lem:sqrtgap}
	Suppose $P, Q$ are symmetric matrices such that both $P$ and $P + Q$ are positive semi-definite.
	Then,
	\[
		\norm{\sqrt{P + Q} - \sqrt{P}} \leq 3\sqrt{\norm{Q}}.
	\]
\end{lemma}
\begin{remark}
	The optimality of the constant $3$
	remains unclear, and it is not unreasonable to speculate that the optimal constant continues to be $1$ even in the
	matrix setting with $P=0$ being the extremal example. 
\end{remark}

We start with an easier version where the smallest eigenvalue of $P$, $\lm(P),$
is assumed to be larger than $\norm{Q},$ i.e., $\lm(P) \geq \norm{Q}$. While
this is not a new result (see e.g., \cite{schmitt1992perturbation}) we provide
a short proof for completeness. 
\begin{lemma}
	\label{lem:singularvaluebnd}
	Let $P, Q$ be two symmetric $d \times d$ matrices such that $\lm(P) \geq \norm{Q}$. Then,
	\[
		\norm{\sqrt{P + Q} - \sqrt{P}} \leq \sqrt{\norm{Q}}.
	\]
\end{lemma}
\begin{proof}
	Let us first recall the well-known fact that the square-root is monotone on the positive semi-definite cone,
	that is, if $0 \preceq P \preceq Q$ then $\sqrt{P} \preceq \sqrt{Q}$, where $\preceq$ denotes the positive semi-definite
	order. A quick proof is provided in the appendix (see Lemma \ref{lem:sqrtmonotone}).

	Invoking this fact and the observation that
	\[
		-\lam I \preceq Q \preceq \lam I, \quad  \text{where } \lam = \norm{Q},
	\]
	we conclude that
	\[
		\sqrt{P - \lam I} - \sqrt{P} \preceq \sqrt{P + Q} - \sqrt{P} \preceq \sqrt{P + \lam I} - \sqrt{P}.
	\]
	(note that $\lm(P) \geq \lam$ is needed for $\sqrt{P - \lam I}$ to be well-defined).
	Thus, for any vector $v$, we have
	\[
		v^T (\sqrt{P - \lam I} - \sqrt{P}) v \leq v^T (\sqrt{P + Q} - \sqrt{P}) v \leq v^T (\sqrt{P + \lam I} -
		\sqrt{P}) v,
	\]
	implying the bound
	\[
		\norm{\sqrt{P + Q} - \sqrt{P}} \leq \max\left\{\norm{\sqrt{P + \lam I} - \sqrt{P}}, \norm{\sqrt{P - \lam I} -
				\sqrt{P}}\right\}.
	\]
	Let us now consider the spectral decomposition $P = VDV^T$ (with $D$ a diagonal matrix), and note that
	\[
		\sqrt{P \pm \lam I} = V \sqrt{D \pm \lam I}\ V^T.
	\]
	Consequently,
	\[
		\norm{\sqrt{P \pm \lam I} - \sqrt{P}} = \norm{\sqrt{D \pm \lam I} - \sqrt{D}} \leq \sqrt{\lam}
	\]
	since $\Abs{\sqrt{a \pm b} - \sqrt{a}} \leq \sqrt{b}$ for reals $a \geq b \geq 0$. This concludes the proof.

\end{proof}

With the above preparation we now prove the general statement.

\begin{proof}[Proof of Lemma \ref{lem:sqrtgap}]
	Let $\del = \norm{Q}$. Then, since $\lm(P + \del I) \geq \norm{Q}$,
	by Lemma \ref{lem:singularvaluebnd} we have
	\[
		\norm{\sqrt{P + Q + \del I} - \sqrt{P + \del I}} \leq \sqrt{\norm{Q}}.
	\]
	(A similar conclusion may be derived from existing results on the square-root gap, see \cite[Lemma
	2.2]{schmitt1992perturbation} for instance. We choose to use our variant instead
	to keep the exposition self-contained.)

	Further, if $P = V D V^T$ is the spectral decomposition of $P$, we have
	\begin{align*}
		\norm{\sqrt{P + \del I} - \sqrt{P}} &= \norm{V \Rnd{\sqrt{D + \del I} - \sqrt{D}} V^T} \\
											&\leq \norm{\sqrt{D + \del I} - \sqrt{D}} \\
											&\leq \sqrt{\del}
	\end{align*}
	since $0 \leq \sqrt{a + b} - \sqrt{a} \leq \sqrt{b}$ for all reals $a, b \geq 0$. A similar bound
	holds for $\norm{\sqrt{P + Q + \del I} - \sqrt{P + Q}}$. Combining these bounds via the triangle inequality,
	we have
	\begin{align*}
		\norm{\sqrt{P + Q} - \sqrt{P}} &\leq \norm{\sqrt{P + Q} - \sqrt{P + Q + \del I}} \\
									   & \phantom{====} + \norm{\sqrt{P + Q + \del I} -
		\sqrt{P + \del I}} + \norm{\sqrt{P + \del I} - \sqrt{P}} \\
									   &\leq \sqrt{\del} + \sqrt{\norm{Q}} + \sqrt{\del} = 3\sqrt{\norm{Q}},
	\end{align*}
	finishing the proof.
\end{proof}

\subsection{Bounding the operator norm of $A_s$}

Returning to the problem at hand, Lemma \ref{lem:sqrtgap} now allows us to bound the gap in \eqref{eq:gap} as
\begin{align}
	\label{eq:gapbound}
	\norm{\G_s - G_s} \les \sqrt{\norm{A_s}}.
\end{align} 
Note that above the positive semi-definite requirement is automatically satisfied by virtue of the definitions in
\eqref{eq:gammagdefs}. We will now justify the claim that $A_s$ is small
in operator norm with high probability via the matrix Bernstein inequality. The version we use here is quoted from \cite[Theorem 5.4.1]{vershynin2018high},
which we restate below.
\begin{lemma}[Matrix Bernstein inequality]
	\label{lem:bern}
	Let $M_1, \ldots, M_n \in \R^{d \times d}$ be independent, centered, symmetric random matrices with $\norm{M_i} \leq
	R$, almost surely. Let $\sig^2 = \norm{\sum_{i = 1}^n \E M_i^2}$. Then, for any $x > 0$, we
	have
	\[
		\P\Rnd{\norm{\sum_{i = 1}^n M_i} \geq x} \leq 2d \cdot \Exp{-\f{3}{2} \cdot \f{x^2}{3\sig^2 + Rx}}.
	\]
\end{lemma}
Applying the lemma with $M_i = \f{1}{n} v_i v_i^T \chi^i_s$ (recall the
definition of $\chi^i_s$ from \eqref{centeredindicator}) which are centered
symmetric matrices will yield the desired bound on $A_s=\sum_{i=1}^n M_i$ (see
\eqref{eq:gap}). Note that for this application
\begin{align*}
	\sig^2 &= \f{1}{n^2}\norm{\sum_{i=1}^n (v_i v_i^T)^2 \cdot \Var(\chi^i_s)} \\
		   &\leq \f{d}{n} \cdot p_s \norm{\f{1}{n}\sum v_i v_i^T} \\
		   &=\f{d}{n} \cdot p_s \norm{U}.
\end{align*}
Here the first inequality uses  that $\Var(\chi^i_s)=p_s(1-p_s) \leq p_s$, and $(v_i v_i^T)^2 \preceq dv_i v_i^T$
($\preceq$ is the positive semi-definite order) for all $i$, on account
of the uniform bound by hypothesis  $\norm{v_i}_2 \leq \sqrt d$.

To determine $R$, we simply note that since $\Abs{\chi^i_s} \leq 1$, one may bound $\norm{M_i} = \f{1}{n} |\chi^i_s| \norm{v_i
v_i^T} \leq \f dn$, and thus we will take $R = \f dn$. Substituting these quantities into Lemma \ref{lem:bern} we obtain the tail bound
\[
	\P\Rnd{\norm{A_s} \geq x} \leq 2d \cdot \Exp{-\f{3}{2} \cdot \f{x^2}{3 \cdot \f{d p_s \norm{U}}{n} + \f{dx}{n}}}.
\]
A simple algebra reveals that 
\begin{equation}\label{normcontrol}
	\P\Rnd{\norm{A_s} \geq x_m} \leq e^{-m},
\end{equation}
where
\[
	x_m \df \max\Rnd{\sqrt{\f{4dp_s \norm{U}}{n}}\Rnd{\sqrt{\log(2d)} + m}, \f{4d}{3n}\Rnd{\log(2d) + m}}.
\]
The straightforward calculations needed to obtain the above expression is recorded in Lemma \ref{lem:tailsimp} applied with $a = 3 \cdot \f{d p_s \norm{U}}{n}, b = \f{d}{n}, c
= 3/2, C = 2d$ (we will need
$d \geq 2$ to satisfy $\log C \geq 1$).

At this point, to ensure notational simplicity, we will assume $\|U\|\ge
\frac{d}{n}$. If that is not the case, one simply needs to replace all the
occurrences of $\|U\|$ by $\norm{U}\vee \frac{d}{n}$ in the remainder of the
argument which we will refrain from doing. 

The above allows us to redefine $x_m$ to be the following larger quantity, for which \eqref{normcontrol} continues to hold, 
\[
	x_m \df 2 \sqrt\f{d \norm{U}}{n}\Rnd{\log(2d) + m}.
\]
That is, for any $m \geq 0$, we have
\begin{align}
	\label{eq:asbound}
	\norm{A_s} \leq 2\sqrt\f{d \norm{U}}{n} \Rnd{\log(2d) + m}, &\quad \text{with probability $\geq 1 - e^{-m}$},
\end{align}
for any $s > 0$.

\subsection{Bounding the quadratic variation}
Armed with the above estimates,  let us return to bounding the quadratic variation by bounding the middle integral in
\eqref{eq:quadvar}. The integral will be split into two parts. The first part will be the integral over $s \leq \ss$
with $\ss \geq 1$ chosen an appropriate poly-logarithmic in $n$ number. On this interval we will bound $\norm{e_1^T (\G_s - G_s)}_2$ by
$\norm{\G_s - G_s}.$ Via \eqref{eq:asbound}, we will apply the concentration in \eqref{eq:asbound} to bound the latter.
This is done since the integral over $[0,\ss]$ will yield a multiplicative
factor $\ss$ which we can only allow to be sub-polynomial in $n.$  The
remainder of the integral will be
treated via direct first moment arguments used to bound $\norm{e_1^T (\G_s - G_s)}_2^2$. While the computations appear
shortly, it is worth remarking at this point that we can control the second
part of the integral using moment arguments since the vectors $v_i$ have
$L^2$
norms bounded by $\sqrt d$, whereas the probability of any $\tau_i$ exceeding $\ss$ is super-polynomially small.

Before proceeding, for convenience, let us recall that
\begin{align*}
	\G_s &= \sqrt{\f{1}{n} \sum_{i = 1}^n v_i v_i^T \1_{\tau_i > s}},
\end{align*}
where $U = \f{1}{n} \sum_{i = 1}^n v_i v_i^T$ and that $G_s^2 = \E \G_s^2 = p_s U$.

Now implementing the above outlined strategy we get
\begin{align}
	[\Phi^1 - \Psi^1]_t &\leq \int_0^\ss \norm{\G_s - G_s}^2 \d s + \int_\ss^t \norm{e_1^T(\G_s - G_s)}_2^2 \d s \nonumber\\
						&\les \int_0^\ss \norm{A_s} \d s +
						d \int_\ss^t \Rnd{\f 1n \sum_{i = 1}^n \1_{\tau_i > s} + p_s} \d s. \label{eq:quadvar2}
\end{align}
Above we use \eqref{eq:gapbound} for the first term and the following crude bound for the second term:
\begin{align*}
	\norm{e_1^T (\G_s - G_s)}_2^2 &\les \norm{e_1^T \G_s}_2^2 + \norm{e_1^T G_s}_2^2 \\
								  &= e_1^T \G_s^2 e_1 + e_1^T G_s^2 e_1 \\
								  &= e_1^T \Rnd{\f1n\sum_{i = 1}^n v_i v_i^T \1_{\tau_i > s}} e_1 + e_1^T p_s U e_1 \\
								  &= \f1n \sum_{i = 1}^n v_{i, 1}^2 \1_{\tau_i > s} + p_s \f1n \sum_{i =
								  1}^n v_{i, 1}^2\\
								  &\leq d\Rnd{\f1n \sum_{i = 1}^n \1_{\tau_i > s} + p_s},
\end{align*}
where we use the assumption that $\norm{v_i}_\infty \leq \norm{v_i}_2\le \sqrt{d}$ for all $i$ in the last step. Since
\[
	\int_{\ss}^\infty \1_{\tau_i > s} \d s = \max(0, \tau_i - \ss),
\]
the bound \eqref{eq:quadvar2} for $t = \infty$ reduces to
\begin{align}
	\label{eq:quadvar3}
	[\Phi^1 - \Psi^1]_\infty &\les \int_0^\ss \norm{A_s} \d s + \f dn \sum_{i = 1}^n \max(0, \tau_i - \ss) +
	d\int_{\ss}^\infty p_s \d s
\end{align}

We will bound each of the terms separately. Note that even though \eqref{eq:asbound} provides control on $\norm{A_s}$,
it is a one-point bound, which needs to be extended to an integral bound. 

{\bf Bounding $\int_0^\ss \norm{A_s} \d s$. } Given the one-point estimate in \eqref{eq:asbound}, we will
		discretize the integral and apply a suitable union bound along with some regularity estimates to extend
		the bound to the entire interval $[0, \ss]$. The regularity estimate
		will rely on showing that the $\tau_i$ are not too crowded, i.e. no
		small enough interval of times contains more than one $\tau_i.$
To show this, it will be useful to assert that the law
		of $\tau_i$ is sufficiently smooth. While one maybe able to construct
		Skorokhod embeddings so that the corresponding stopping time $\tau$ has
		a reasonably smooth distribution, the requirements for us are rather
		mild, and we will achieve it by slightly perturbing the law of our noise
		variables $a_i$ instead.
		
\noindent		
\textit{Density bound:}		
		Let $a'_i \df a_i + \eps_i$ with $\eps_i = \pm n^{-r}$ uniformly at random (independent
		of $a_i$), with $r$, a large enough constant chosen
		later. Crucially, one may choose a Skorokhod embedding of $\eps_i$, say $\tau^{\eps_i}$ which is continuous (given
		by the first hitting time of $\pm n^{-\eps}$ for Brownian motion), and thus, one may Skorokhod-embed $a'_i$ as
		$\tau'_i = \tau_i + \tau^{\eps_i}$, a continuous random variable with density bounded
		by $O(n^{2r})$. That $\tau^{\eps_i}$ has density bounded by
		$O(n^{-2r})$ is immediate since it has the same distribution as
		$n^{-2r}$ times the first hitting time of $\pm 1$
			for a Brownian motion and the latter has a bounded density.   The same bound transfers over to
			$\tau'_i$ since convolution of two random variables one of whose
			densities is bounded by $O(n^{2r})$, yields a random variable with
			the same bound on the density regardless of the distribution of the
			other random variable.

We next show that indeed the perturbation of the noise variables can be performed without affecting the sought coupling between $a$ and $g.$			
			Since,
		\begin{align}
			\label{eq:mollerr}
			n^{-\f12}\norm{\sum_{i = 1}^n v_i a_i - \sum_{i = 1}^n v_i a'_i}_\infty
			\leq n^{-\f12} \sum_{i = 1}^n \norm{v_i}_\infty \cdot \Abs{\eps_i} \leq n^{-(r - 1)},
		\end{align}
		(since $\norm{v_i}_\infty\le \sqrt d \le \sqrt n$ by hypothesis)
		we may replace $a_i$ by $a'_i$ (and $\tau_i$ by $\tau'_i$), which we do in the sequel, with a
		(deterministic) additional coupling error of at most $n^{-(r - 1)}$.
		
		However, having replaced $a_i$ by $a'_i$, to match variances, we now replace $g_i$ by $g'_i$, where
		$g'_i$ is Gaussian with
		$\Var(g'_i) = \Var(g_i) + \Var(\eps_i) = 1 + n^{-2r}$. 
		Since $\TV{g_i - g'_i} \leq
		O\F{\Var(g'_i) - \Var(g_i)}{\Var(g_i)} = O(n^{-2r})$ (this follows from a straightforward computation,
		but for a concrete reference, see \cite[Corollary 2]{barsov1987estimates}), 
		$g_i$ and $g'_i$ can be
		coupled to be equal with probability $\geq 1 - O(n^{-2r})$. A union bound over all the variables show that
		$\{g_i\}_i$ and $\{g'_i\}_i$ can be coupled such that
		\begin{equation}
			\label{eq:mollerr2}
			S^g_n = S^{g'}_n, \quad \text{with probability $\geq 1 - O(n^{-r})$},
		\end{equation}
		when $r \geq 1$ (we will account for this additional probability error in the final bound).

Having the luxury of the above density bound, let us continue with the discretization partitioning $[0, \ss]$ into $N \leq \ss n^k$ intervals
		$I_1, \ldots, I_N$, each of
		length at most $n^{-k}$ (for some $k \geq 1$ to be chosen later). For a given interval $I_j$,
		\[
			\P(\text{at least two $\tau'_i$ land in $I_j$})  \les n^{2 + 4r - 2k},
		\]
		and by a union bound,
		\[
			\P(\text{there is some interval with at least two $\tau'_i$}) \les N \cdot n^{2 + 4r - 2k} \les \ss \cdot n^{2 + 4r - k}.
		\]
		With a choice of, say, $k = 5r + 3$, the probability is at most $O(n^{-r})$. 

		The bound \eqref{eq:asbound} on $\norm{A_s}$ implies via a union bound that
		with probability $\geq 1 - O(e^{-m})$,
		\begin{align*}
			\sup_{s \in \partial} \norm{A_s} & \les \sqrt\f{d \norm{U}}{n} \Rnd{\log(2d) + \log (N + 1) + m} \overset{d \leq n, \ss \geq 1}{\les} \sqrt{\f{d \norm{U}}{n}} \Rnd{k \log n + \log \ss + m},
		\end{align*}
		where $\partial$ denotes the $N + 1$ boundary points of the intervals $\{I_j\}$.
		To extend this supremum throughout $[0, \ss]$, we will now use that fact that with probability $1 - O(n^{-r})$,
		no interval $I_j$ has two $\tau'_i$s. Let $I_j = [s^-, s^+]$ be one of the intervals. Then for any $s \in [s^-,
		s^+]$,
		\[
			\norm{A_s - A_{s^-}} = \norm{\f1n \sum_{i = 1}^n v_i v_i^T \chi^i_s - \f1n \sum_{i = 1}^n v_i v_i^T
			\chi^i_{s^-}}.
		\]
		At most one $\1_{\tau'_i > s}$ differs from $\1_{\tau'_i > s^{-}}$, so 
		\[
			\norm{A_s - A_{s^-}} \leq \f1n \sup_i \norm{v_i v_i^T} + \Abs{p_{s^-} - p_{s}} \les \f{d}{n} + n^{-r},
		\]
		since by the density bound of $\tau'_i$ stated above,  $p_{s^-} - p_s \leq \P(s^- < \tau'_i \leq s^+) \les n^{2r-k} \les n^{-r}$. Since $r, d \geq 1$, we have
		\[
			\norm{A_s - A_{s^-}} \les \f{d}{n}.
		\]
		
		Since $\norm{U}\ge \frac{d}{n}$ (recall that we have been assuming that since the beginning of this argument for simplicity), with probability $\geq 1 - O(e^{-m} +
		n^{-r})$ we have
		\begin{align}
			\label{eq:term1}
			\int_0^\ss \norm{A_s} \d s & \les \ss \cdot \Rnd{ \sqrt{\f{d\norm{U}}{n}} \Rnd{k \log n + \log \ss + m}}.
		\end{align}
We now move on to the second term in \eqref{eq:quadvar3}.

{\bf Bounding $\f1n \sum_{i = 1}^n \max(0, \tau_i - \ss)$. }
Note that since we replaced $\tau_i$ by $\tau'_i$ in the previous step, the quantity we are interested in is
\[
	\f1n \sum_{i = 1}^n \max(0, \tau'_i - \ss).
\]
The expectation of this quantity is
\[
	\E\max(0, \tau'_i - \ss) \leq \E\max(0, \tau_i - \ss) + \E\tau^{\eps_i} \leq \E[\tau_i \1_{\tau_i \geq \ss}] + O(n^{-2r}).
\]
By Cauchy-Schwarz inequality we have
\begin{align}
	\label{eq:cs1}
	\E[\tau_i \1_{\tau_i \geq \ss}] \leq \sqrt{\E[\tau_i^2] \cdot \E[\1_{\tau_i \geq \ss}]} \les
	\Exp{-\Om(\ss^\beta)},
\end{align}
due to the stretched-exponential tails of $\tau$, see Remark \ref{rem:setails} for the precise definition of 
$\beta$.

	Therefore, by Markov's inequality, with probability $1 - n^{-r}$,
\[
	\f1n \sum_{i = 1}^n \max(\tau'_i - \ss, 0) \les \f{\Exp{-\Om(\ss^\beta)} + O(n^{-2r})}{n^{-r}}.
\]

Let us choose $\ss$ of order $\log^{1/\be} n$ such that $\Exp{-\Om(\ss^\beta)} = O(n^{-2r})$. Then, with probability $1 - n^{-r}$,
\[
	\f1n \sum_{i = 1}^n \max(0, \tau'_i - \ss) \les n^{-r}.
\]

We now arrive at the final term in \eqref{eq:quadvar3}.

\noindent
{\bf Bounding $\int_{\ss}^\infty p_s \d s$. }
One may rewrite this term as (again recalling that $\tau_i$ is now $\tau_i'$)
		\begin{align}
			\int_{\ss}^\infty p_s \d s &= \int_{\ss}^\infty \P(\tau_i' > s) \d s \\
									 &= \E[(\tau'_i - \ss) \cdot \1_{\tau'_i > \ss}]\\
									 &\leq \E[\tau_{i'} \cdot \1_{\tau_{i'} > \ss}]\\
									 &\leq O(n^{-2r}),
		\end{align}
		using \eqref{eq:cs1} and the choice of $\ss$ above.

Combining the bounds established above, we may convert \eqref{eq:quadvar3} to read (still $a'$ replacing $a$ and $g'$ replacing $g$)
\begin{align}
	\label{eq:termcomb}
	[\Phi^1 - \Psi^1]_\infty &\les \ss \cdot \Rnd{\sqrt\f{d\norm{U}}{n} \Rnd{k \log n + \log \ss + m}} + n^{-r+1},
\end{align}
with probability $\geq 1 - O(e^{-m} + n^{-r})$, for every $m \geq 0$ (we used $d n^{-r} \le n^{-r+1}$ above). Recalling our choice $\ss = O(\log^{1/\be} n)$,
and using $m = r \log n$ and that $\norm{U}\ge \f dn$, we
may simplify the bound to
\[
	[\Phi^1 - \Psi^1]_\infty \les \log^{\f1\be} n \cdot \Rnd{\sqrt\f{d\norm{U}}{n} \cdot r \log n},
\]
with probability $\geq 1 - O(n^{-r})$, whenever $r \geq 1$, since $k = 5r + 3$.

In the next lemma (Lemma \ref{lem:quadvartail}) we record the well known result bounding the  value of a martingale in
terms of its quadratic variation. Invoking this we obtain the following bound on $|\Phi^1 - \Psi^1|$:
\[
	\Abs{\Phi^1_\infty - \Psi^1_\infty} \les r \log^{1 + \f1\be} n \cdot {\F{d\norm{U}}{n}^{1/4}},
\]
with probability $\geq 1 - O(n^{-r})$. A union bound over all the $d$ coordinates then produces the same bound
\begin{align*}
	\norm{\Phi_\infty - \Psi_\infty}_\infty &\les r \log^{1 + \f1\be} n \cdot {\F{d\norm{U}}{n}^{1/4}},
\end{align*}
but with probability $\geq 1 - O(n^{-r + 1})$. At this point, recall the coupling error incurred due to the replacement
of the $\{a_i\}$ by $\{a'_i\}$, and of $\{g_i\}$ by $\{g'_i\}$ in \eqref{eq:mollerr} and \eqref{eq:mollerr2}
respectively. Incorporating this error (and returning to $S^a_n$ and $S^g_n$), we have
\[
	\norm{S^a_n - S^g_n}_\infty \les r \log^{1 + \f1\be} n \cdot {\F{d\norm{U}}{n}^{1/4}} + n^{1 - r},
\]
with probability $1 - O(n^{-r + 1}) - O(n^{-r}) = 1 - O(n^{-r + 1})$. 
Choosing $r\ge 4$, the term $n^{1 - r}$ can be shown to be smaller than the first term in the above error bound.
Replacing $r$ by $2r$ finishes the proof.
\qed

\begin{remark}\label{optimalexponent}
Note that we had remarked right after the statement of Theorem \ref{thm:main}
about the sub-optimality of the exponent $1/4$ in the error bound
$\F{d\norm{U}}{n}^{1/4}$ and that we expect it to be $1/2.$ The source of the
$1/4$ stems from the square root error bound in Lemma \ref{lem:sqrtgap}. As
discussed preceding the latter lemma, the bound can be improved to have a
linear dependence of $\norm{Q},$ if the smallest eigenvalue of $P$ is large
enough, which we expect to the case for us. However, we do not pursue this and
getting the optimal error bound is left as an open problem.  
\end{remark}

The following classical result used above quantifies bounds on the tail of a
martingale in terms of its quadratic variation. The straightforward proof is
provided in the appendix.

\begin{lemma}
	\label{lem:quadvartail}Let $B_t$ be a Brownian motion on $[0, \infty)$ equipped with a natural filtration and $V_t$
	be an {adapted} process on $[0, \infty)$. Define $X_t$ via $\d X_t = V_t \d B_t$ with
	$X_0 = 0$. Then, if there exists $M, \eps > 0$ such
	\[
		\P([X]_\infty > M) < \eps,
	\]
	then
	\[
		\P\Rnd{\sup_{t} |X_t| > \sqrt{2M\log(C/\eps)}} < 2\eps,
	\]
	for a universal constant $C$.
\end{lemma}

\begin{remark}
	\label{rmk:dncounter}
	Note that the bound in Theorem
	\ref{thm:main} is effective only when $d\ll n.$ Thus, one may wonder about the plausibility of such a statement for
	larger values of $d.$  Intuitively it is perhaps clear that the dimension $d$ needs to be smaller than the
	number of samples to lead to a CLT and here we present a quick example with $d=n$ where indeed a
	statement of the kind of Theorem
	\ref{thm:main} cannot possibly hold.  Let $v_1,v_2,\ldots, v_n$ form the columns of a DFT matrix of size $n$, say
	$U.$ Then $U U^\dag=I.$ Now let $g$ and $a$ be vectors of i.i.d. Gaussians and random variables distributed according
	to $\mu$ respectively. Then for any coupling of $g$ and $a,$ note that $\|U(g-a)\|^2_2= \|g-a\|^2_2.$ Hence,
	$\E\|U(g-a)\|^2_2\ge n W^2_2(g,a),$ thereby indicating that the $L^\infty$ norm of $U(g-a)$ cannot be small
	with high probability. 
\end{remark}

\def\tR{\widetilde{R}_{n, \g}}
\def\tRi{\widetilde{R}_{\infty, \g}}
\def\tZa{\tZ_{n, \g, a}}
\def\tZg{\tZ_{n, \g, g}}

We now prove Theorem \ref{thm:gmc} using as input Theorem \ref{thm:main}.

\section{Application to the multiplicative chaos: Proof of Theorem \ref{thm:gmc}}\label{gmcpf}

We start by  briefly recalling the proof strategy outlined in Section \ref{iop} in a bit more detail before presenting formal arguments. Given that we
will essentially follow the steps from \cite{kk}, to keep things brief and avoid repetition, we will mostly focus on
indicating the key steps and highlighting the essential differences from the proof of \cite{kk}.
The interested reader is referred to \cite{kk} for a more detailed exposition.

\bul {\bf Reduction to a hierarchical model.} The first step is to discretize the continuous model of the random
Fourier series corresponding to variables $a$ (and $g$) to a \emph{discrete hierarchical model} (defined in
\cite[Section 2.1]{kk}). This reduces things to a finite dimensional problem. Hierarchical models such as branching random walk have long served as tractable proxies for
more complicated log-correlated fields such as the Gaussian Free Field.  Thus, this discretization step can be viewed as
a counterpart in the analysis of the random Fourier series model. 

 The hierarchical model is also log-correlated  and has a
corresponding multiplicative chaos (again by the results of \cite{junnila2020multiplicative} in the non-Gaussian case), denoted $\muta$ (and
$\mutg$). The effectiveness of this approximation lies in the fact that $\mua$ and $\muta$ are mutually absolutely
continuous, for any $\g \in (0, \sqrt 2)$ (see \cite[Proposition 2.3]{kk}).
We will slightly simplify the hierarchical construction
and present it below, but the key point is that it now suffices to show that $\muta$ and $\mutg$ are mutually absolutely continuous. 

Crucially, the finite dimensional reduction is what will allow us to apply the CLT result Theorem \ref{thm:main} to
compare the hierarchical models for the non-Gaussian variables $a$ and Gaussians $g$.

However, as will be apparent shortly, the number of vertices in the hierarchical model at level $m$ is $2^m$ (with
$\poly(m)$ corrections) corresponding to the partial sum in \eqref{fourier1} for $k\in [2^{m-1}, 2^m)$ which
involves $2^m$ independent random variables. Thus, in the notation of Theorem \ref{thm:main}, $d\approx n \approx
2^m$ which presents a difficulty since the effectiveness of Theorem \ref{thm:main} relies on $d$ being much smaller
than $n$. The key observation in \cite{kk} which goes back to the earlier alluded to property of the fractal nature of
the GMC, is that it suffices to compare the two measures on their effective
supports at scale $m$ whose sizes are in fact much smaller than $2^m$. These
are given by the thick points allowing us
to take $d$ to be polynomially small compared to $n$, which we turn to next.

\bul {\bf Thick points suffice.} The
multiplicative chaos measures are supported on
\emph{$\g$-thick points} of the corresponding processes. Postponing more formal
definitions to later, any point $t$ in the support of $\mua$ must satisfy
$S_{n,a}(t)\approx m \g \log 2$ for all large $m.$
This continues to hold in the hierarchical models as well (this is a general feature of log-correlated processes).

 Further, such ``thick vertices'' are polynomially rare, in the sense that the collection of $\g$-thick vertices at level $m$, i.e.,
\[
	\cT_{m, \g, a} \df \{ v : S_{n, a}(v) \geq m \g \log 2 \},
\]
where $v$ is a vertex at level $m$, has cardinality of the order of $2^{m \Rnd{1 - \f{\g^2}{2}}}$ (see
\eqref{eq:szk} for a bound on a related object) up to
$\poly(m)$ corrections, which follows by a first moment computation.

\bul {\bf Coupling at thick points.} 
Since at level $m,$ the number of thick vertices is polynomially (in $2^m$) smaller than the number of variables which is of order $2^m,$  Theorem
\ref{thm:main} may be invoked to couple the Gaussian and non-Gaussian values so that the gap between $S_{n,a}$ and $S_{n,g}$ is
uniformly small (in fact, summably
so in $m$) on all the thick vertices, with high probability. This is achieved in Lemma \ref{lem:coupling2}. A key
component of this step is the observation that the vectors $v_i$ used in this application are columns of a DFT
matrix of size $2^m \cdot \poly(m)$, and thus the corresponding matrix $U$ may be shown to be sufficiently isotropic (Lemma
\ref{lem:normu}) which is crucial to the effectiveness of Theorem \ref{thm:main}.

\bul {\bf Proof of absolute continuity.} After the coupling is produced, it remains to show that under this
coupling, $\muta$ and $\mutg$ are indeed mutually absolutely continuous.
Informally, the reason behind this is that under the coupling, on the supports
of the chaos measures, $S_{\infty,a}-S_{\infty,g}$ is $O(1).$ The formal
argument for this part is carried out 
in \cite{kk} and this part of their argument is $\g$-independent and hence works for the entire regime $(0,\sqrt 2)$. The 
Radon-Nikodym derivative between $\muta$ and $\mutg$ is constructed by first
considering the densities between the measures constructed using the
hierarchical process at level $m$, which are
piecewise-constant functions and pose no definitional issues. These densities
can then be shown to converge to a limit with the limit then shown to be a
valid density
between $\muta$ and $\mutg$.
Some more details are spelled out after the proof of Lemma \ref{lem:coupling2}, but the reader is encouraged to
refer
to \cite{kk} for the full proof.

\subsection{The hierarchical model}\label{hierarchical1}
We now precisely define the hierarchical model, closely following \cite[Section 2.1]{kk}.

\bul First define a sequence of (almost dyadic) partitions $\cP_n$ of $[0, 1]$ via
\[
	\cP_n \df \Bra{ \f{q}{2^n f(n)} : q = 0, \ldots, 2^n f(n)}, \quad n \geq 1,
\]
where $f(n) = n^4$. The partition $\cP_n$ corresponds to level $n$ of the
hierarchical model, which is ``responsible'' for frequencies $2^{n - 1}, \ldots, 2^n - 1$ of the Fourier series. As in
\cite{kk}, the additional polynomial factor $f(n)$ is necessary to ensure that the meshing is finer than the frequency
scale.

\bul Define the mesh points at level $n \geq 1$, recursively, as
\begin{align}\label{mesh23}
	\cS_n \df \cS_{n - 1} \cup \cP_n, \quad \cS_0 = \{0, 1\},
\end{align}
and let $\tau_n(1) < \tau_n(2) < \ldots < \tau_n(|\cS_n|)$ be the ordering of $\cS_n$, and let $I_n(j)$ be the interval
$[\tau_n(j), \tau_n(j + 1))$, where $j < |\cS_n|$. 

Note that the partition $\{I_{n + 1}(\cdot)\}$ is a refinement of the partition $\{I_n(\cdot)\}$, producing a
\emph{hierarchical structure} consisting of intervals $I_n(j)$ at level $n$, with children $I_{n + 1}(j')$ at level $n +
1$,
such that $I_{n + 1}(j') \sse I_n(j)$. Abstractly, we will refer to nodes of this tree by $v$, and the corresponding
interval as $I(v)$. Denoting the nodes at level $n$ as $\cN_n$, for each node $v \in \cN_n$ we will define $t(v)$ as
the element in $\cP_n$ which is the left endpoint of $I(v)$.

\bul For each node $v \in \cN_n$, we define the contribution per vertex 
\begin{align}\label{block1}
	X_a(v) &= \sum_{k = 2^{n - 1}}^{2^n - 1} \f1{\sqrt{k}} \Rnd{\ao_k \cos(2\pi k t(v)) + \at_k \sin(2\pi k t(v))}, \\
	X_g(v) &= \sum_{k = 2^{n - 1}}^{2^n - 1} \f1{\sqrt{k}} \Rnd{\go_k \cos(2\pi k t(v)) + \gt_k \sin(2\pi k t(v))},
\end{align}
and let the partial sums be defined as
\begin{equation}\label{sumhier}
	\tS_a(v) = \sum_{w \preceq v} X_a(w), \quad \tS_g(v) = \sum_{w \preceq v} X_g(w),
\end{equation}
where $w \preceq v$ denotes that $w$ is an ancestor of $v$ in the tree (note that this is equivalent to $I(v) \sse
I(w)$).

Using these partial sums, one may define a measure $\mutan$ on $[0, 1]$ with a piecewise constant density via
\begin{align}
	\label{eq:mutdef}
	\mutan(\d t) \df \f{e^{\g \tS_a(t)}}{Z_{n, \g, a}} \d t, \quad \tZ_{n, \g, a} \df \E e^{\g \tS_a(t)},
\end{align}
where, by slight abuse of notation,
\begin{align}
	\label{eq:sdef2}
	\tS_a(t) = \tS_a(v), \quad \text{for } t \in I(v).
\end{align}

As discussed earlier, results from \cite{junnila2020multiplicative} show that the measures
$\mutan$ converge weakly a.s. to a measure $\muta$ for which we have the following result.
\begin{lemma}[{\cite[Proposition 2.3]{kk}}]
	\label{lem:hier}
	For any $\g \in (0, \sqrt 2)$, the measures $\mua$ and $\muta$ are mutually absolutely continuous.
\end{lemma}

In the sequel, we will therefore focus on proving the following, which, combined with Lemma \ref{lem:hier} will
complete the proof of Theorem \ref{thm:gmc}.
\begin{lemma}[{Analog of \cite[Proposition 2.7]{kk}}]
	\label{lem:hier2}
	There is a coupling $\Bra{\ao_k, \at_k}_{k \geq 1}$ and $\Bra{\go_k, \gt_k}_{k \geq 1}$ such that
	$\muta$ and $\mutg$ are mutually absolutely continuous.
\end{lemma}

A natural approach is to construct this coupling hierarchically,
i.e., at each step $n$, to couple $\Bra{\ao_k, \at_k}_{k = 2^{n - 1}}^{2^n - 1}$ with $\Bra{\go_k, \gt_k}_{k
= 2^{n - 1}}^{2^n - 1}$, such that $\tS_a(v)$ and $\tS_g(v)$ are close for all $v \in \cN_n$.

However, as alluded to earlier, the number of leaves at level $n$ is $O(2^n)$ (up to $\poly(n)$ corrections), while the
number of random variables involved in \eqref{block1} is $O(2^n)$ which by the
discussion in Remark \ref{rmk:dncounter} poses a serious obstruction to the
existence of such a coupling.
To circumvent this
difficulty, the key
observation (as already noted in \cite{kk}) is that it suffices to construct a
coupling such that $\tS_a(v)$ and $\tS_g(v)$ are close at vertices $v$ which
are thick points. More precisely, to ensure that the thickness of a vertex only
uses information from past scales, we will consider the coupling at a vertex
$v$ only if their parent, denoted $\pi(v)$ is $(\g - \del)$-thick (in either
$S_a$ or $S_g$). Summarizing,  we will attempt to couple at $S_a(v)$ and
$S_g(v)$ if and only if $\pi(v) \in \cT_{n - 1, \g, a} \cup \cT_{n - 1, \g,
g}$ where $\cT_{\ell, \g, a}$ is the collection of $(\g - \del)$-thick points at level $\ell$ for $\tS_a$:
\[
	\cT_{\ell, \g, a} \df \{ v \in \cN_\ell : S_a(v) \geq \ell \Rnd{(\g - \del) \log 2} \}.
\]
Let us refer to these ``children of thick parents'' as $\cK_{n, \g, a}$. 
We will also set 
\begin{equation}\label{thickpointset}
\cKn = \cKa \cup \cKg.
\end{equation}
As in the Gaussian case, which was known earlier, see e.g., \cite{rhodes2014gaussian},  it was shown in \cite[Lemma 3.1]{kk} that the multiplicative chaos measure $\muta$ is
supported on the set of points which are \emph{$(\g - \del)$-thick}, i.e., all sufficiently deep
ancestors are $(\g - \del)$-thick (while sharper results are expected to be true and are known in the Gaussian case, for our purposes a small $\delta$ slack will suffice). 

To quickly motivate the $\log 2$ factor, note that for any $v \in \cN_n$,
the contribution to the variance of $\S_a(v)$ from the ancestor $w \in \cN_\ell$ at any level $\ell \leq n$ is:
\begin{align}
	\label{eq:svar}
	\sum_{k = 2^{\ell - 1}}^{2^\ell - 1} \f1{k} \Rnd{\cos^2(2\pi kt) + \sin^2(2\pi kt)} = \log 2 + O(2^{-\ell}),
\end{align}
and thus, $\Var(S_a(v)) = n \log 2 + O(1)$. It may be helpful to the reader to
note that this suggests that the setting may be compared to a branching random
walk where each edge
of the binary tree has an associated standard Gaussian variable of variance
$\log 2.$ Thick points are those whose values are a constant multiple of the
maximum value.

We now arrive at our main coupling result for any subset $\cK \sse \cN_n$ (throughout, one should think of $\cK$
as the set of vertices at level $n$ with $(\g-\del)$-thick parents, as discussed above). The proof is a direct application of our CLT in Theorem \ref{thm:main}.
A related statement appeared in  \cite[Lemma 2.5]{kk},
which stipulated essentially that $|\cK| \le 2^{\frac{n}{2}}.$ 
As already alluded to, this was proven as an immediate consequence of a high
dimensional CLT result which goes back to \cite{y}. The latter result was
covariance agnostic and hence our improvement allowing us to take $|\cK|$ as
large as $2^n$ is crucially reliant on showing that the relevant covariance
matrix has small spectral norm.

\begin{lemma}[Coupling]
	\label{lem:coupling}
	Fix $n \geq 1$. For any subset $\cK \sse \cN_n$ such that $|\cK| \leq 2^n$, there is a coupling of
	$\Bra{(\ao_k, \at_k)}_{k = 2^{n - 1}}^{2^n - 1}$ and $\Bra{(\go_k, \gt_k)}_{k = 2^{n - 1}}^{2^n - 1}$ such that,
	\[
	\P\Rnd{\sup_{v \in \cK} \Abs{X_a(v) - X_g(v)} > \poly(n) \F{\Abs{\cK}}{2^n}^{1/4}} \les 2^{-n},
	\]
	where $\poly(n)$ denotes a fixed polynomial in $n$.
\end{lemma}
\begin{proof}
	\def\kl{{k^{-}}}
	\def\ku{{k^{+}}}
	\def\po{\psi^{(1)}}
	\def\pt{\psi^{(2)}}
	\def\fo{\phi^{(1)}}
	\def\ft{\phi^{(2)}}

	We will first rephrase the problem in terms of Theorem \ref{thm:main}. Let us write $\kl = 2^{n - 1}, \ku = 2^n
	- 1$ (the lower
	and upper bounds on $k$, respectively) and define
	the random vectors $\po_k, \pt_k \in
	\R^\cK$ as
	\begin{align}\label{vec234}
		\po_k \df \sqrt{\f{\kl}{k}} \fo_k, \quad & \pt_k \df \sqrt{\f{\kl}{k}} \ft_k, \quad \text{ where $\fo_k, \ft_k
		\in \R^\cK$ satisfy} \\
		\Rnd{\fo_k}_v \df \cos(2\pi k t(v)), \quad & \Rnd{\ft_k}_v \df \sin(2\pi k t(v)), \quad v \in \cK, \quad k \in [\kl,\ku].	\end{align}
	With this notation,
	\[
		X_a =\frac{1}{\sqrt{k^-}} \sum_{k = \kl}^\ku \ao_k \po_k + \at_k \pt_k,
	\]
	where we think of $X_a$ as a vector  in $\R^\cK$ mapping $v \to X_a(v)$. A similar expression also holds for
	$X_g$.

	Note that by definition, since $k^-\le k $, we have $\norm{\po_k}_\infty,
	\norm{\pt_k}_\infty \leq 1$ and hence their $L^2$ norms bounded by $\sqrt{|\cK|}$. Since the number of random variables is $2
	\cdot (\ku - \kl + 1) = 2^n= 2k^-$,
	the conditions of Theorem \ref{thm:main} are satisfied by the random vector $\frac{X_a}{\sqrt{2}}.$ The corresponding matrix $U$ is then
	\[
		U = 2^{-n} \sum_{k = \kl}^\ku \po_k \Rnd{\po_k}^T + \pt_k \Rnd{\pt_k}^T.
	\]
	However, since $\kl \leq k$, we have that $U \preceq V$ where
	\[
		V = 2^{-n+1} \sum_{k = \kl}^\ku \fo_k \Rnd{\fo_k}^T + \ft_k \Rnd{\ft_k}^T.
	\]
	Consequently, $\norm{U} \leq \norm{V}$, and it suffices to bound $\norm{V}$. At this point we will crucially use
	the fact that the matrix $V$ bears a strong resemblance to a DFT matrix allowing us to control its spectral norm
	via the upcoming lemma. To compactify the expression for $V$, note that if we define complex vectors
	$\zeta_k \in \bC^\cK$ as
	\begin{equation}\label{matnot12}
		(\zeta_k)_v \df \exp(2\pi \i k t(v)), \quad v \in \cK, \i = \sqrt{-1},
	\end{equation}
	then for any $v, w \in \cK$ we have
	\begin{align*}
		\Re{\zeta_k \zeta^\dag_k}_{vw} &= \Re{\Exp{2\pi \i k t(v) - 2\pi \i k t(w)}} \\
									   &= \cos(2\pi k t(v) - 2\pi k t(w)) \\
									   &= \cos(2\pi k t(v)) \cos(2\pi k t(w)) + \sin(2\pi k t(v)) \sin(2\pi k t(w)) \\
									   &= \Rnd{\fo_k \Rnd{\fo_k}^T}_{vw} + \Rnd{\ft_k \Rnd{\ft_k}^T}_{vw}.
	\end{align*}
	Thus, $\Re{\zeta_k \zeta^\dag_k} = \fo_k \Rnd{\fo_k}^T + \ft_k \Rnd{\ft_k}^T$, and hence,
\begin{equation}\label{vmat}
		V = \Re{2^{-n} \sum_{k = \kl}^\ku \zeta_k \zeta^\dag_k}.
\end{equation}
	Applying the following Lemma \ref{lem:normu}  we obtain
	\[
		\norm{U} \leq \norm{V} \les n^{10}.
	\]
	Therefore, by Theorem \ref{thm:main} with $r = 1$ and $\al = 1$ (and thus $\be = 1/3$), we have that
	\[
		\sup_{v \in \cK} \Abs{X_a(v) - X_g(v)} \les
			\log^{4/3} (2^n) \Rnd{\F{\Abs{\cK} n^{10}}{2^n}^{1/4} +
			\F{\Abs{\cK}}{2^n}^{1/2}} \les \poly(n) \F{\Abs{\cK}}{2^n}^{1/4},
	\]
	with probability $\geq 1 - O(2^{-n})$.
\end{proof}

\subsection{Spectral norm bounds using the Fourier structure}
\def\dag{\dagger}
\begin{lemma}
	\label{lem:normu}
	For any subset $\cK \sse \cN_n$ such that $|\cK| \leq 2^n$ as in Lemma \ref{lem:coupling}, for $V$ as in \eqref{vmat},
 $$\norm{V} \les n^{10}.$$
\end{lemma}
\begin{proof}
Let $f(n)=n^4$ as in the hierarchical construction.
Note that any $\cK$ in the hypothesis corresponds to points  
 $0 \le t_1 < t_2 < \ldots < t_m \le 1$  such that for each $i$, there is an $\ell \leq n$
	with $t_i \cdot 2^\ell \cdot f(\ell) \in \Z$. {Note that
	since all the $t_i$s are distinct, we have $m \leq \sum_{i = 1}^n 2^i \cdot f(i) = 2^{n + 1} \cdot f(n)$}.
Since the vectors $\zeta_k$ in \eqref{matnot12} are indexed by $k\in [k^-,k^+],$ to simplify the index set,
define the vectors $v_0, \ldots, v_{2^{n-1} - 1} \in \mathbb{C}^m$ by $v_k=\zeta_{k'}$ where $k' = k + 2^{n-1}.$ 

For notational convenience, we will replace $n$ by $n+1$ so that there are $2^{n}$ vectors in total and not
$2^{n-1}$ and proceed to bound the spectral norm of 
		\[
		W = \Re{\f{1}{2^n}\sum_{k = 0}^{2^n - 1} v_k v_k^\dag}.
	\] 

It will be convenient to partition the points $t_i$ into layers and analyze the
contributions of each part to the spectral norm. The partition will involve two
index parameters. One corresponding to which level of the hierarchical tree the
point first appeared. The second corresponds to the shift at that level needed
to make it a dyadic point. Note that each level $\ell$ introduces $2^\ell
\ell^4$ many new points. Thus, there are at most $\ell^4$ possible shift values
$r$ for points appearing first at level $\ell$.

	For each $\ell \leq n$, define
	\[
		J_\ell = \{ 1 \leq j \leq m : \ell \text{ is smallest such that } t_j \cdot 2^\ell \cdot f(\ell) \in \Z\},
	\]
	and then partition $J_\ell = \cup_{r = 0}^{f(\ell) - 1} L_\ell^r$ via
	\[
		L_\ell^r = \{ j \in J_\ell : (t_j \cdot 2^\ell \cdot f(\ell))\mod f(\ell) = r\}, \quad 0 \leq r < f(\ell).
	\]
	One may partition each vector $v_k$ as
	\[
		v_k = \sum_{\ell, r} v_{k, \ell, r},
	\]
	where $v_{k, \ell, r} \in \bC^{m}$ inherits the values of $v_k$ on the indices in $L_\ell^r$ (with other
	entries set to $0$).
	Then,
	\begin{align*}
		2^n W &= \Re{\sum_{k, \ell, r, \ell', r'} v_{k, \ell, r} v_{k, \ell', r'}^\dag} \\
			  &= \sum_{\ell, r, \ell', r'} \Re{\sum_k v_{k,\ell,r} v_{k,\ell',r'}^\dag}
	\end{align*}
	For any vector $\psi \in \R^m$ the quadratic form is then
	\begin{align}
		2^n \cdot \psi^T W \psi &= \sum_{\ell, r, \ell', r'} \sum_k \psi^T \Re{v_{k, \ell, r} v_{k, \ell', r'}^\dag}
		\psi \\
								& {\les} \sum_{\ell, r, \ell', r'} \sum_k \Abs{\Inn{\psi, v_{k, \ell, r}}} \cdot \Abs{\Inn{\psi,
								v_{k, \ell', r'}}}\\
								&\leq (n f(n))^2 \cdot \sup_{\ell, r, \ell', r'} \sum_k\Abs{\Inn{\psi, v_{k, \ell, r}}}
								\cdot \Abs{\Inn{\psi, v_{k, \ell',
								r'}}}  \\
								&\les (n f(n))^2 \cdot \sup_{\ell, r, \ell', r'} \sum_k \Abs{\Inn{\psi, v_{k, \ell, r}}}^2 +
								\Abs{\Inn{\psi, v_{k, \ell', r'}}}^2 \\
								&\les (n f(n))^2 \cdot \sup_{\ell, r} \sum_k \Abs{\Inn{\psi, v_{k, \ell, r}}}^2
								\label{eq:normua}
	\end{align}
	where in the first inequality we use the fact that if $v, w \in \mathbb{C}^m$, then
	\begin{align*}
		\psi^T \Re{v w^\dag} \psi &= \psi^\dag \F{v w^\dag + w v^\dag}{2} \psi \\
								  &= \Inn{\psi, v}\bar{\Inn{\psi, w}} + \Inn{\psi, w}\bar{\Inn{\psi, v}} \\
								  &\les \Re{\Inn{\psi, v}\bar{\Inn{\psi, w}}} \\
								  &\leq \Abs{\Inn{\psi, v}} \cdot \Abs{\Inn{\psi, w}}
	\end{align*}
	since $\Re{a\bar{b} }\leq |a|\cdot|b|$ for complex $a, b$. For a fixed $\ell, r$, 
	\begin{align*}
		\sum_k \Abs{\Inn{\psi, v_{k, \ell, r}}}^2 &= \psi^T \Rnd{\sum_k v_{k, \ell, r} v_{k, \ell, r}^\dag} \psi
	\end{align*}
	To finish the proof, in light of \eqref{eq:normua}, it will thus suffice to show that 
	\begin{align*}
		\norm{\sum_{k} v_{k, \ell, r} v_{k, \ell, r}^\dag} \leq 2^n.
	\end{align*}
	Observe that any entry $(j, j')$ of the matrix above is $0$ if either $j$ or $j'$ is not in $L_\ell^r$. Thus,
	fix $j, j' \in L_\ell^r$, and note that
	\begin{align}
		\Rnd{\sum_k v_{k, \ell, r} v_{k, \ell, r}^\dag}_{jj'} &= \sum_k (v_{k, \ell, r})_j \bar{(v_{k, \ell, r})_{j'}}
		\\
															  &= \sum_k \Exp{2\pi \i \cdot k' \cdot (t_j - t_{j'})}
															  \label{eq:normub},
	\end{align}
	where recall $k' = k + 2^n$.
	Since $j, j' \in L_\ell^r$, we have
	\[
		(t_j - t_{j'}) \cdot 2^\ell \cdot f(\ell)  \mod f(\ell) = 0.
	\]
	Thus, $(t_j - t_{j'}) \cdot 2^\ell \eqcolon q \in \Z$. Then \eqref{eq:normub} is
	\begin{align*}
		\sum_{k = 0}^{2^n - 1} \Exp{2\pi \i \cdot (k + 2^n) \cdot \f{q}{2^\ell}} &\overset{\text{since } \ell \leq n}{=} 
		\sum_{k = 0}^{2^n - 1} \Exp{2\pi \i
		\cdot \f{kq}{2^\ell}} = 2^n \cdot \1_{q = 0},
	\end{align*}
	by usual geometric sums of roots of unity. Thus, 
	\[
		\sum_k v_{k, \ell, r} v_{k, \ell, r}^\dag = 
			\begin{cases}
				2^n \cdot \1_{j = j'} & j, j' \in L_\ell^r,\\
				0 & \text{otherwise},
			\end{cases}
	\]
	It is clear that this matrix has norm exactly $2^n$ (unless $L_\ell^r$ is empty, when it is 0), concluding the
	proof.
\end{proof}

\subsection{Combining the pieces}

Applying Lemma \ref{lem:coupling} to $\cK = \cK_{n, \g}$ defined in \eqref{thickpointset}, we obtain the following corollary, which is the analog of
\cite[Lemma 2.6]{kk}.
\begin{lemma}
	\label{lem:coupling2}
	There is a coupling of $\Bra{\ao_k, \at_k}_{k \geq 1}$ and $\Bra{\go_k, \gt_k}_{k \geq 1}$ such that
	for any $n \geq 1$,
	\[
		\sup_{v \in \cKn} \Abs{X_a(v) - X_g(v)} \leq \poly(n) \cdot 2^{-\f{n (\g - \del)^2}{8}},
	\]
	with probability $\geq 1 - O(n^{-2})$.
\end{lemma}
In fact, as will be apparent, the failure probability can be made exponentially
small in $n$ by changing a few parameters but for our purposes only their
summability will be important.
\begin{proof}
	Let us first estimate the size of $\cKn$. Recalling \eqref{sumhier}, note that,
	\begin{align}
		\E\Abs\cKn  &\leq  \cdot \poly(n) \cdot \Abs{\cN_{n - 1}} \cdot \left[\sup_v \P(\tS_a(v) \geq (n - 1) \Rnd{(\g - \del) \log 2})\right. \\
		& \left.\quad \quad +\sup_v \P(\tS_g(v) \geq (n - 1) \Rnd{(\g - \del) \log 2}) \right], 
				 \label{eq:kn1}
	\end{align}
where the supremum is over $v \in \cN_{n - 1}$, and the $\poly(n)$ is a crude
	upper bound for the branching factor of the tree. In fact both the
	probability terms above will be essentially the same owing to the one
	dimensional central limit theorem. Formally, note that
	\[
		\tS_a(v) = \sum_{w \preceq v} X_a(w) = \sum_{k = 1}^{2^{n - 1} - 1} \f1{\sqrt{k}} \Rnd{\ao_k \cos(2\pi k t(w_k))
		+ \at_k \sin(2\pi k t(w_k))},
	\]
	where $w_k$ is the depth $\ell$ ancestor of $v$ with $\ell$ defined via $2^{\ell - 1} \leq k < 2^\ell$ (in
	particular, $v$ is the depth $n - 1$-ancestor of itself since it is at
	level $n - 1$).  The following precise expression of the Laplace transform
	was obtained in \cite[Lemma 6]{junnila2020multiplicative} 
	\begin{equation}\label{laplace}
		\E \Exp{(\g - \del) \tS_a(v)} = \Exp{\f{(\g - \del)^2}{2} \E[\tS_a(v)^2] + O_\g(1)},
	\end{equation}
	 where $O_\g(1)$ denotes a term bounded
	as a function of $\g$ (and not $n$). This establishes a form of the one
	dimensional CLT which will suffice to establish that the probability of $v$
	being a thick point is at an exponential scale the same as in the Gaussian
	case. Note that in the latter, one does not have the $O(1)$ correction in
	\eqref{laplace}.
	
	Using a variance computation (as in \eqref{eq:svar}), we obtain that
	\[
		\E[\tS_a(v)^2] = (n - 1)\log 2 + O(1),
	\]
	and thus
	\[
		\E \Exp{(\g - \del) \tS_a(v)} = \Exp{\f{(\g - \del)^2}{2} (n - 1)\log 2 + O_\g(1)}.
	\]
	The probability that $v$ is $(\g - \del)$-thick is then, by Markov's inequality,
	\begin{align*}
		\P\Rnd{\tS_a(v) \geq (n - 1) \Rnd{(\g - \del) \log 2}} &= \P\Rnd{\Exp{(\g - \del) \tS_a(v)} \geq \Exp{(\g -
	\del)^2 (n - 1) \log 2}}\\
															   &\les \Exp{(n - 1)\log 2 \Rnd{\f{(\g - \del)^2}{2} - (\g - \del)^2}}
		\\
															   &= 2^{-(n - 1)\f{(\g - \del)^2}{2}}.
	\end{align*}
	The Gaussian case has the same proof from \eqref{laplace} onwards.
	Substituting this into \eqref{eq:kn1} we obtain
	\begin{equation}
		\label{eq:szk}
		\E\Abs\cKn \les \poly(n) \cdot 2^{n \Rnd{1 - \f{(\g - \del)^2}{2}}},\
	\end{equation}
	and thus, again by Markov's inequality, we have that
	\[
		\P\Rnd{\Abs\cKn \geq \poly(n) \cdot 2^{n \Rnd{1 - \f{(\g - \del)^2}{2}}}} \leq n^{-2}, \quad n \geq 1,
	\]
	where $\poly(n)$ refers to a fixed polynomial in $n$.  Call this event $E$.  Note that $E$ does not reveal the
	noise determining $X_a(v)$ and $X_g(v)$, and thus we may condition on it and apply Lemma \ref{lem:coupling} to
	obtain
	\[
		\sup_{v \in \cKn} \Abs{X_a(v) - X_g(v)} \les \poly(n) \F{\Abs{\cKn}}{2^n}^{1/4} \les \poly(n) \cdot 2^{-\f{n (\g -
		\del)^2}{8}}
	\]
	with probability $\geq 1 - O(2^{-n})$. Therefore, the unconditional probability is at least $(1 - O(2^{-n}))(1 -
	\P(E)) \geq 1 - O(n^{-2})$, finishing the proof.
\end{proof}

\begin{remark}
	Note that the bound in Lemma \ref{lem:coupling2} is effective as soon as $\g > 0$ (for sufficiently small
	$\del$), while the corresponding bound in \cite[Lemma 2.6]{kk} is effective only when $\g > 1$. This is the key
	improvement which allows our result to extend to the entire range of $\g \in (0, \sqrt 2)$. The main reason for
	the improvement is that Theorem \ref{thm:main} quantifies the dependence of the coupling error on the ``isotropy'' of the
	vectors, as encoded by $\norm{U}$. Observe that, in the notation for Lemma \ref{lem:coupling}, the trivial bound
	on $\norm{U}$ is $\tr(U)$ which is $O(\Abs\cK)$, since the individual vectors have norms on the order of
	$\sqrt{\Abs\cK}$. This bound would have led to a coupling error of $\F{\Abs{\cK}^2}{2^n}^{1/4}$
	(up to $\poly(n)$ factors). This recovers the results of \cite{kk}, since the bound is small only when $\cK \ll
	2^{n/2}$, something that happens for $\cKn$ only when $\g > 1$. The
	explicit CLT result used in \cite[Lemma 2.4]{kk} which goes back to
	\cite{y} was covariance agnostic. Thus, explicitly pinning down the role of
	$U$ in the coupling
	error in Theorem \ref{thm:main}, and exploiting the DFT structure of $U$ to obtain a $\poly(n)$ bound on
	$\norm{U}$ in Lemma \ref{lem:normu}, are crucial in obtaining a bound that is effective for all $\g > 0$.
\end{remark}

At this point, only the last missing piece in the proof of Lemma
\ref{lem:hier2} outlined before Section \ref{hierarchical1} remains. This part
is verbatim \cite[Section 3]{kk}. We
discuss the broad steps here, and refer the reader to \cite{kk} for more details.

Define for each $t \in [0, 1]$,
\begin{align*}
	\tR(t) &\df \frac{\Exp{\g \tS_g(v)} \big/ \tZg}{\Exp{\g \tS_a(v)} \big/ \tZa},
\end{align*}
using the notation from \eqref{eq:mutdef} and \eqref{eq:sdef2}. Note that a natural candidate for the Radon-Nikodym
derivative $\d\mutg/\d\muta$ is the limit of $\tR$ as $n \to \infty$. The proof for this claim is divided into three steps: Almost surely (in the noise variables, i.e, $\P$-a.s.),
{\bf (A) Convergence.} $\tR(t)$ converges to  a value in $(0, \infty)$ for $\muta$-a.e. $t$, as well as for
$\mutg$-a.e. $t$.

{\bf (B) Density behavior.} For all Borel $A \sse [0, 1]$, we have
\begin{align*}
	\lim_{n \to \infty} \int_{A} \tR(t) \muta(\d t) &= \mutg(A), \\
	\lim_{n \to \infty} \int_{A} \tR(t)^{-1} \mutg(\d t) &= \muta(A).
\end{align*}

{\bf (C) Uniform integrability.} The family $\{\tR(\cdot)\}_{n \geq 1}$ is uniformly integrable with respect to $\muta$ and
$\{\tR(\cdot)^{-1}\}_{n \geq 1}$ is uniformly integrable with respect to $\mutg$.

The proof of Lemma \ref{lem:hier2}
given the above steps follows from standard analysis arguments. Step A yields a limiting
Radon-Nikodym derivative, say $\tRi$, and the convergence is in $L^1$ due to uniform integrability (Step C).
Further, uniform integrability allows limits to commute with integration in Step B, yielding a valid limiting derivative. 

Establishing the claims in steps A-C is technically more involved and we refer to \cite[Section 3]{kk} for the details.

We end by discussing higher dimensional generalizations which was also touched upon in \cite{kk}. 

\subsection{Higher dimensional multiplicative chaos}\label{gdom} 

While the main object in the paper was a log-correlated process on the unit interval,
as was already observed in \cite{kk}, the results have natural higher dimensional analogues. For simplicity, we
consider the unit cubes. For $D=[0,1]^d,$ let
$\lambda_k$ be the $k
^{th}$ smallest
eigenvalue (with multiplicity) of $-\Delta$ (the Laplacian) in $D$ with Dirichlet boundary conditions, and let $h_k$ denote the
corresponding eigenfunction, normalized so that $\|h_k\|_{L^2(D)}=1$. Given this we can consider the following series analogous to \eqref{fourier1}, given by
\begin{equation}\label{highdim24}
S_{n, g}(x) \df \frac{1}{(2\pi)^d}\sum_{k = 1}^n g_k \lambda_k^{-\frac{d}{4}} h_k(x), \quad x \in D,
\end{equation}
where $g = (g_k)_{k\in \N}$ denote a sequence of i.i.d. standard Gaussian
random variables. One can similarly define $S_{n, a}(x)$ for a generic i.i.d.
sequence of mean zero, variance one random variables $a = (a_k)_{k\in \N}.$ The
$S_{n, g}(\cdot)$ process was considered in \cite{log} termed as the
log-correlated Gaussian process on $D$. The associated multiplicative chaos
measure is then $\mu_{\gamma,g}$. The analogous $a$ counterparts make sense
following \cite{junnila2020multiplicative} (this is because we are restricting
to the case of the cubes and for more general domains one needs certain
regularity properties of the eigenfunctions $h_k$). The measures are
non-degenerate now whenever $\gamma \in (0,\sqrt{2d})$ with the $L^2$ regime
being $\gamma \in (0,\sqrt{d}).$ The arguments in this paper generalize
verbatim to prove that $a$ and $g$ can be coupled so that $\mu_{\gamma,g}$ and
$\mu_{\gamma,a}$ are mutually absolutely continuous to each other under the
coupling for all $\gamma\in (0,\sqrt{2d})$ while as in the $d=1$ case, the
arguments in \cite{kk} cover the regime $(\sqrt d, \sqrt{2d})$.

Analogous to \eqref{block1}, one may now consider the blocks  
\begin{align}
	\label{block2}
	X_a(x) &= \sum_{\lambda_k \in (2^{2(n - 1)}, 2^{2n}]} \lambda_k^{-d/4} a_k h_k(x), \\
	X_g(x) &= \sum_{\lambda_k \in (2^{2(n - 1)}, 2^{2n}]} \lambda_k^{-d/4} g_k h_k(x),
\end{align}

The arguments from here on are particularly simplified by the fact that the eigenfunctions $h_k$ are the tensor
product of $d$ one dimensional eigenfunctions. That is, $$h_k(x_1,x_2,\ldots,x_d)=f_1(x_1)\otimes f_2(x_2)
\otimes\cdots \otimes f_d(x_d)$$ where for any $1\le i\le d$, $f_i(x_i)=\sin(2\pi k_ix_i) \text{ or }\cos(2\pi k_i x_i),$
for some non-negative integers $k_i$ with the corresponding $\lambda_k=\sum_{i=1}^d k_i^2.$ It follows that
$\lambda_{k} \approx k^{\frac{2}{d}}$ and hence the above sums have about $2^{nd}$ terms in them. 
{
Thus, we can follow the one dimensional strategy and define the hierarchical model with a mesh which is a $d$-fold
tensor product of the one dimensional mesh constructed in \eqref{mesh23} allowing to apply the CLT result  with vectors
$w_i$ now admitting again a tensor structure $v_1\otimes v_2 \otimes\cdots \otimes v_d$ where $v_j$ are the vectors
appearing in \eqref{vec234}. The only remaining thing worth pointing out then is that the corresponding matrix
$\displaystyle{W=\frac{1}{2^{nd}}\sum_{k: \lambda_k \in (2^{2(n - 1)}, 2^{2n}]}h_kh_k^T}$ is then in a
positive definite sense upper bounded by the matrix 
$$
\widetilde W = \underbrace{U\otimes U \otimes \cdots \otimes U}_{d \text{ times}}
$$
where $U=\frac{1}{2^n}\sum_{j=1}^{2^n}\big[u_{j}u_j^T+v_{j}v_j^T\big]$ where $u_j$ and $v_j$ are given by the restrictions of
$\sin(2\pi j x)$ and $\cos(2\pi j x)$ on the $2^n n^4$ mesh points. Eigenvalues of $\widetilde W$ are now products of
eigenvalues of $U$ hence implying that $\|\widetilde W\|=\|U\|^d$. Since $\|U\|$ is sub-polynomial in $n$ by the arguments
as in Lemma \ref{lem:normu}, so is $\|W\|.$

}

\section{Appendix}

We finish with the outstanding proofs of some of the statements which have featured in our arguments.

\begin{proof}[Proof of Lemma \ref{lem:diagonalization}]
	The result is clear if $Q_t$ is always full rank, since via 
	Levy's characterization (\cite[Theorem 3.16]{shreve}) that a $d$-dimensional Brownian
	motion is the only continuous martingale with quadratic variation $It$, one may check that $Q_t^{-1} \d X_t$
	is a Brownian motion. However, the invertibility fails in general, so we start by considering the ``completion''
	$W_t$ defined via
	\[
		\d W_t = Q_t^+ \d X_t + \sqrt{I - Q_t^+ Q_t} \d B'_t, \quad Y_0 = 0,
	\]
	for an independent Brownian motion $B'_t$, where $Q_t^+$ is the pseudo-inverse of $Q_t$. Using Levy's characterization
	it is easy to check that $W_t$ is a Brownian motion since
	\begin{align*}
		\d [W, W]_t &= Q_t^+ \d [X, X]_t \Rnd{Q_t^+}^T + (I - Q_t^+ Q_t) \d t \\
				  &\overset{(a)}{=}  Q_t^+ Q^2_t (Q_t^+)^T \d t + (I - Q_t^+ Q_t) \d t \\
				  &= I \d t.
	\end{align*}
	Step $(a)$ above invoked the fact that the quadratic variation of $X$ is
	\[
		\d [X, X]_t = A_t A_t^T \d t = Q_t^2 \d t.
	\]
	To finish the proof, we will show that $dX_t = Q_t \d W_t$ by showing that the difference process
	$Y_t$ satisfying $\d Y_t = \d X_t - Q_t \d W_t$ has zero quadratic variation (implying that $Y_t \equiv 0$, 
	since it is also a continuous martingale, and starts at zero). Note that
	\begin{align*}
		\d Y_t = \d X_t - Q_t \d W_t &= \d X_t - Q_t \Rnd{Q_t^+ \d X_t + \sqrt{I - Q_t^+ Q_t} \d B'_t}  \\
						&= (I - Q_t Q_t^+) A_t \d B_t - \Rnd{Q_t \sqrt{I - Q_t^+ Q_t}} \d B'_t
	\end{align*}
	and thus the quadratic variation is
	\begin{align*}
		\d [Y, Y]_t = \Box{ (I - Q_t Q_t^+) A_t A_t^T (I - Q_t^+ Q_t) + Q_t (I - Q^+_t Q_t) Q_t }  \d t.
	\end{align*}
	Direct computations show that both terms above are zero, which finishes the proof. 
	To see this for the first term, recall that $Q_t^2 = A_t A_t^T$, implying
	\begin{align*}
		A_t A_t^T (I - Q_t^+ Q_t) = Q_t^2 - Q_t^2 Q_t^+ Q_t = 0
	\end{align*}
	since $MM^+ M = M$ for any positive semi-definite matrix $A$. For the second term, we have
	\begin{align*}
		Q_t (I - Q_t^+ Q_t) Q_t = Q_t^2 - Q_t Q_t^+ Q_t^2 = 0,
	\end{align*}
	again for the same reasons.
\end{proof}

\begin{lemma}[Square-root is monotone]
	\label{lem:sqrtmonotone}
	If $P \preceq Q$ for two symmetric $d \times d$
	positive semi-definite matrices (with $\preceq$ denoting the positive semi-definite order),
	then $\sqrt{P} \preceq \sqrt{Q}$.
\end{lemma}
\begin{proof}
	It is cleaner (and equivalent) to show that if $P^2 \preceq Q^2$ then $P \preceq Q$.	Further,
	by a perturbation argument, one may reduce to the case when $P$ is strictly positive definite.

	For any vector $v \in \R^d$, we have
	\[
		v^T (Q - P) v = v^T P^{\f12} (P^{-\f12} Q P^{-\f12} - I) P^{\f12} v 
	\]
	so that $P \preceq Q \iff I \preceq P^{-\f12} Q P^{-\f12}$ (when $P$ is invertible).
	In terms of eigenvalues, this is the same as $\lm(P^{-\f12} Q P^{-\f12}) \geq 1$, which, by similarity,
	is equivalent to $\lm(QP^{-1}) \geq 1$.

	Further, by the positive semi-definiteness of $Q$, we have that
	\[
		v^T P^{-\f12} Q P^{-\f12} v = (P^{-\f12} v)^T Q (P^{-\f12} v) \geq 0,
	\]
	so that $P^{-\f12} Q P^{-\f12}$ is positive semi-definite, implying that $\lm(QP^{-1}) \geq 0$.

	By exactly the same reasoning, the provided condition $P^2 \preceq Q^2$ is equivalent to
	$\lm(P^{-1} Q^2 P^{-1}) \geq 1$. But
	\[
		P^{-1} Q^2 P^{-1} = (QP^{-1})^T (QP^{-1}),
	\]
	so $\lm(P^{-1} Q^2 P^{-1}) \geq 1 \implies \lm(QP^{-1}) \geq 1$, since \emph{a priori}, we have
	that $\lm(QP^{-1}) \geq 0$. This can be seen by considering the quadratic
	form of $v$, the eigenvector of $QP^{-1}$ corresponding to $\lm(QP^{-1})$
	(the existence of which is guaranteed by the similarity of $QP^{-1}$ to the
	symmetric matrix $P^{-\f12} Q P^{-\f12}$) with respect to $ (QP^{-1})^T
	(QP^{-1}).$ This concludes the proof.
\end{proof}

The following simple result contains a computation that allows us to derive effective bounds from Bernstein's inequality.

\begin{lemma}
	\label{lem:tailsimp}
	Suppose a random variable $X \geq 0$ satisfies
	\[
		\P(X \geq x) \leq C \cdot \Exp{-\f{c x^2}{a + bx}}, \quad x \geq 0,
	\]
	for constants $a, b, c, C> 0$ with $\log C \geq 1$.  Then, for all $m > 0$,
	\[
		\P(X \geq x_m) \leq e^{-m},
	\]
	where
	\begin{align*}
		x_m &\df \max(x_{1, m}, x_{2, m}) \\
		x_{1, m} &\df \sqrt{\f{2a}{c}} \Rnd{\sqrt{\log C} + m} \\
		x_{2, m} &\df \f{2b}{c} \Rnd{\log C + m} .
	\end{align*}
\end{lemma}
\begin{proof}
	To start, note that $a + bx \leq 2 \cdot \max(a, bx)$ implying
	\[
		\f{cx^2}{a + bx} \geq \min\Rnd{\f{cx^2}{2a}, \f{cx}{2b}}.
	\]
	If $x \geq \sqrt{2a/c} \Rnd{\sqrt{\log C} + m} = x_{1,m}$, then
	\[
		\f{cx^2}{2a} \geq \Rnd{\sqrt{\log C} + t}^2 \geq \log C + m,
	\]
	since $\log C \geq 1$. If $x \geq (2b/c) \Rnd{\log C + m} = x_{2,m}$, then
	\[
		\f{cx}{2b} \geq \log C + m.
	\]
	Thus, if $x \geq x_m = \max(x_{1, m}, x_{2, m})$, then
	\[
		\P(X \geq x_m) \leq C \cdot \Exp{-\f{c x^2}{a + bx}} \leq C \cdot \Exp{-(\log C + m)} \leq e^{-m},
	\]
	as required.
\end{proof}

We end this section with the brief proof of Lemma \ref{lem:quadvartail}.
\begin{proof}[Proof of Lemma \ref{lem:quadvartail}]
	The proof follows from a direct application of the Dubins-Schwarz theorem \cite[Theorem 4.6]{shreve}, which states that for any continuous
	martingale $X_t$, there is a Brownian motion $W$ (on a potentially extended space) such that $X_t = W_{[X]_t}$. Let $A$ be the event
	that $[X]_\infty \leq M$. On $A$, 
	\[
		\sup_t |X_t| \leq \sup_{s \leq M} |W_s|.
	\]
	Standard Brownian estimates imply
	\[
		\P\Rnd{\sup_{s \leq M} |W_s| > x} \les \P(|W_M| > x) \les \Exp{-x^2 / 2M}.
	\]
	Consequently, 
	\[
		\P\Rnd{\sup_t |X_t| > \sqrt{2M \log(C/\eps)}} \leq \P(A^c) + \eps \leq 2\eps,
	\]
	for a universal constant $C$.
\end{proof}

\appendix

\bibliographystyle{alpha}
\bibliography{ref}

\end{document}